\documentclass[11pt]{amsart}
\usepackage{a4}
\usepackage{amssymb}
\usepackage{comment}

\usepackage{color}
\makeatletter
\newcommand{\leqnomode}{\tagsleft@true}
\newcommand{\reqnomode}{\tagsleft@false}
\makeatother

\newtheorem{theorem}{Theorem}[section]
\theoremstyle{plain}
\newtheorem{lemma}[theorem]{Lemma}

\newtheorem{corollary}[theorem]{Corollary}

\newtheorem{definition}[theorem]{Definition}

\numberwithin{equation}{section}

\DeclareMathOperator*{\esssup}{ess\,sup}

\newcommand\trnorm[1]{\left\|\kern-1.2pt\left|#1\right\|\kern-1.2pt\right|}

\newcommand{\N}{\mathbb{N}}
\newcommand{\C}{\mathbb{C}}
\newcommand{\R}{\mathbb{R}}
\DeclareMathOperator*{\bd}{bd}
\DeclareMathOperator*{\dist}{dist}
\DeclareMathOperator*{\supp}{supp}
\DeclareMathOperator*{\esssupp}{ess\,supp}

\newcommand{\comments}

\let\latexchi\chi
\makeatletter
\renewcommand\chi{\@ifnextchar_\sub@chi\latexchi}
\newcommand{\sub@chi}[2]{
  \@ifnextchar^{\subsup@chi{#2}}{\latexchi^{}_{#2}}%
}
\newcommand{\subsup@chi}[3]{
  \latexchi_{#1}^{#3}%
}
\makeatother
\begin{document}

\title[Density of smooth functions in $W^{1,\Phi}(\Omega)$]{Density of compactly supported smooth functions $C_C^\infty(\mathbb{R}^d)$ in  Musielak-Orlicz-Sobolev spaces $W^{1,\Phi}(\Omega)$}



\author{Anna Kami\'{n}ska}
\address{Department of Mathematical Sciences,
The University of Memphis, TN 38152-3240}
\email{kaminska@memphis.edu}

\author{Mariusz \.Zyluk}
\address{Pozna\'n, Poland}
\email{mzyluk@gmail.com}

\date{\today}

\thanks{}

\begin{abstract}
We investigate here the density of the set of the restrictions from $C_C^\infty(\R^d)$ to $C_C^\infty(\Omega)$ in the Musielak-Orlicz-Sobolev space $W^{1,\Phi}(\Omega)$. It is a continuation of  article \cite{KamZyl3}, where we have studied density of $C_C^\infty(\R^d)$ in $W^{k, \Phi}(\R^d)$ for $k\in\N$. The main theorem states that  for an open subset $\Omega\subset \R^d$ with its boundary  of class $C^1$, and  Musielak-Orlicz function  $\Phi$  satisfying  {\rm condition (A1)} which is a sort of log-H\"older continuity and the growth condition  $\Delta_2$,  the set of restrictions of functions from $C_C^\infty(\R^d) $ to $\Omega$ is dense in $W^{1,\Phi}(\Omega)$. We obtain a corresponding result in variable exponent Sobolev space $W^{1,p(\cdot)}(\Omega)$ under the assumption that the exponent $p(x)$ is essentially bounded on $\Omega$ and  $\Phi(x,t) = t^{p(x)}$, $t\ge 0$, $x\in\Omega$, satisfies the log-H\"older condition.  

\end{abstract}

\maketitle

 This article is the second part of  work \cite{KamZyl3} investigating the problems of density of smooth functions in Musielak-Orlicz-Sobolev (MOS) spaces $W^{k,\Phi}(\Omega)$. While in the first part we have  proved that the space $C_C^\infty(\R^d)$ of smooth functions compactly supported in open subset $\Omega\subset \mathbb{R}^d$ is dense in  $W^{k,\Phi}(\R^d)$  under regularity conditions (A1) and $\Delta_2$ on $\Phi$, in this current article  we will establish the density of restrictions of the functions from $C_C^\infty(\R^d)$ to $C_C^\infty(\Omega)$ in $W^{1,\Phi}(\Omega)$ under similar assumptions on $\Phi$ and additional mild assumptions that the boundary of $\Omega$ is of class $C^1$. We assume here that $k=1$ since the whole process   requires  substantially more technical effort than the previous case \cite{KamZyl3}.

In the first section we provide basic notations and facts on Musielak-Orlicz (MO) spaces as well as Musielak-Orlicz-Sobolev (MOS) spaces which is based on the information contained in \cite{KamZyl3}. The second section is the major part of the article where we are proving through the effort of several technical lemmas the main result Theorem \ref{main result of 6}. 

This article is in fact  a different version of the result   
from \cite{Ahmida},  where the authors   have shown that under  the segment condition of the boundary of $\Omega$ (see \cite[Definition 3.21, p. 68]{Adams}), and local integrability of $\Phi$, $\Delta_2$ and $\mathcal{M}2$ condition,    the set of restrictions of functions from $C_C^\infty(\R^d)$ to $C_C^\infty(\Omega)$ is dense in $W^{k,\Phi}(\Omega)$.  As was discussed in the first part \cite{KamZyl3}, condition $\mathcal{M}2$ is rather strong and for instance is not covering the case of variable exponent functions. In section \ref{section 2} we establish a version of the mentioned theorem  for $k=1$ under weaker assumptions, which includes variable exponent functions, Corollary \ref{cor: result of 6}.  

 \begin{section}{Introduction}\label{section 1}
For a natural number $d\in\mathbb{N}$, we denote the $d$-dimensional Euclidean space by $\R^d$.  If  $x=(x_1,\dots,x_d)\in\R^d$ then  $|x|$ is its Euclidean norm.
 By $I_n$, $n=1, \dots, d$,  we denote any interval  $(a_n,b_n), \ [a_n,b_n), \ (a_n,b_n]$ or $ [a_n,b_n]$, where $-\infty<a_n<b_n< \infty$. A cuboid is a set of the form
 $R=I_1\times I_2\times\dots\times I_d.$
For any  point $(x_1,\dots,x_d)=x\in\R^d$ and $l>0$ we define the open cube of center $x$ and side-length $l$ as
\[
Q(x,l)=\left\{(y_1,\dots,y_d)\in\R^d: \ |x_n-y_n|<\frac{l}{2}, \text{ for }n=1,\dots,d \right\}.
\]
Similarly for any  point  $x\in \R^d$ and any $r>0$ we define the open ball of center $x$ and radius $r>0$ as
$B(x,r)=\{y\in\R^d: \ |x-y|<r\}.$
A general open ball  or open cube often will be given without specifying their centers and radii or side lengths. In those cases an open ball will be usually denoted by $B$ and an open cube as $Q$.

For any subset $A\subset \R^d$ we denote by $A^\circ$ its interior and by $\overline{A}$ its closure. As usual  for any set $A\subset \R^d$ we define its boundary as
$\bd(A)=\overline{A}\setminus A^\circ= \overline{A}\cap\overline{A^c}.$

 If $A\subset \R^d$ is a Lebesgue measurable set, then its Lebesgue measure will be denoted as $|A|$. Clearly, for any cuboid $R=I_1\times\dots\times I_d \subset \R^d$, where $I_1,\dots I_d$ are finite intervals, we have
 $|R|=|\overline{R}|=|I_1|\cdots |I_d|$,
 where $|I_n|$ stands for the length of the interval $I_n$, $n=1,\dots,d$. Similarly, for any $x\in \R^d$ and $l,r>0$ we have
$|Q(x,l)|=|\overline{Q(x,l)}|=l^d$ and
$|B(x,r)|=|\overline{B(x,r)}|=\sigma_d r^d$,
where 
\begin{align}\label{1}
 \sigma_d=\begin{cases} \frac{\pi^n}{n!} & d=2n, \ n\in \N \\
\frac{2(n!)(4\pi)^n}{(2n+1)!} & d=2n+1, \ n\in \N.
    \end{cases}   
\end{align}

  \end{section}

For a Lebesgue measurable set $A\subset \mathbb{R}^d$, the set of all Lebesgue measurable complex valued functions  $f:A \to \C$ will be denoted as $L^0(A)$.  Given $f\in L^0(A)$ its support is defined as the set  $\supp(f)=\{x\in A: f(x)\neq 0\}$.
 If $\Omega\subset \R^d$ is open  and $f\in L^0(\Omega)$ then we define the essential support of $f$ as 
$\esssupp (f)=\Omega\setminus\bigcup \{U\subset \Omega: U\text{ is open and } f=0 \text{ a.e. on } U\}.$
 Clearly $\esssupp (f)$ is a closed subset of $\Omega$.
 
For an open set $\Omega\subset \R^d$ the symbol $C^k(\Omega)$ stand for the set of all $k$-times continuously differentiable functions defined on $\Omega$. 
The set of all smooth functions defined on $\Omega$, that is functions having all derivatives,  will be denoted by $C^\infty(\Omega)$.  Clearly if  $f\in C(\Omega)$ then 
$\esssupp (f)=\overline{\supp(f)}.$
 By $C_C^\infty(\Omega)$ we denote the set of all smooth compactly supported functions defined on $\Omega$.

Given $f,g\in L^0(\R^d)$,  the convolution $f*g$ is given by the formula
\[
f*g(x)=\int\limits_{\R^d}f(y)g(x-y)dy,
\]
for $x\in\R^d$ for which the integral exists.
 A function $f\in L^0(\Omega)$ is said to be {\it locally integrable} if $\int\limits_K|f(x)|dx<\infty$ for any compact subset $K$ of $\Omega$.
We denote the set of all locally integrable functions on $\Omega$ by $L^1_{loc}(\Omega)$. 
The restriction of a function $f:\R^d\to\C$ to $\Omega$ is defined as
\[f_{|\Omega}(x)=\begin{cases}f(x) & x\in\Omega \\
0 & x\notin\Omega .\end{cases}\]

As usual, a partial derivative of a function $f$ at a point $x$ with respect to the variable $x_i$ will be denoted as $\partial_i f(x)$, $i=1,\dots,d$.
A multi-index $\alpha=(\alpha_1,\dots,\alpha_d)$ is an ordered $d$-tuple of non-negative integers. We adopt the following notations. 
\begin{enumerate}
    \item[$(1)$]  $|\alpha|= \alpha_1+\dots+\alpha_d$,
    \item[$(2)$]  $\alpha!=\alpha_1!\cdot\dots\cdot\alpha_d!$,
    \item[$(3)$]  $\partial^\alpha=(\partial_1)^{\alpha_1}\dots(\partial_d)^{\alpha_d}$.
    \item[$(4)$]  For $\alpha=(\alpha_1,\dots,\alpha_d)$, $\beta=(\beta_1,\dots,\beta_d)$ we write $\alpha\leq\beta$ if $\alpha_i\leq\beta_i$ for $i=1,\dots,d$.
\end{enumerate}

Now we introduce the notion of the {\it Musielak-Orlicz (MO) function}.

\begin{definition}\label{def}
Let $A\subset \R^d$ be a measurable subset of $\mathbb{R}^d$, a function $\Phi:A\times [0,\infty )\to [0,\infty )$ is called a Musielak-Orlicz function {\rm(}$MO$ function{\rm)} if
\begin{enumerate}
\item[$(a)$]  for every $x\in A$, $t\mapsto\Phi(x,t)$ is convex, 
\item[$(b)$] for every $x\in A$, $\Phi(x,t)=0$ if and only if $t=0$,
\item[$(c)$] for every $ t\in [0,\infty), \ x\mapsto\Phi (x,t)$ is  Lebesgue measurable on $A$.
\end{enumerate}
\end{definition}
We say that $MO$ function $\Phi$ is {\it locally integrable} on $A$, if for every compact set $K\subset A$ and every $\lambda>0$, $\int\limits_K\Phi(x,\lambda)dx<\infty$.
We also say that  $\Phi$  satisfies  {\it $\Delta_2$ condition}, a growth condition, if there exist a constant $C>0$ and a positive function $h\in L^1(A)$, that is $\int_A h(x)\, dx <\infty$, such that 
\[
 \   \Phi(x,2t)\leq C\Phi(x,t)+h(x), \ \ \ \ t\geq 0,\ \text{ a.a.} \ x\in  A.
 \]
    
 Recall that  MO function $\Phi$ defined on an open subset $\Omega\subset\R^d$, satisfies {\it condition} {\rm{(A1)}} if   there exist constants $\beta,\delta \in (0,1)$ such that for all open balls $B$ with $|B|<\delta$ and almost all $x,y\in B\cap \Omega$,
 \[
\Phi(x,\beta t)\leq \Phi(y,t),\ \ \text{whenever}\ \  t\in \left[\Phi^{-1}(y,1) , \Phi^{-1}\left(y,\frac{1}{|B|}\right)\right].
\]
If $\Phi$ satisfies (A1)  and $\Delta_2$  then  it is locally integrable on $\Omega$  (Lemma 4.7 in \cite{KamZyl3}).  The  condition (A1),  which is an analogue of log-H\"older  condition  for the variable exponent MO functions \cite{KamZyl3}, will be crucial in further considerations.

For  a $MO$ function $\Phi$ on $A$ define the functional
\[I_\Phi(f)=\int\limits_{A}\Phi(x,|f(x)|)dx\]
for every $f \in L^0(A)$, and  let 
     $L^\Phi(A)$ be  {\it Musielak-Orlicz space} ($MO$) space associated with  $\Phi$, that is 
    \[
    L^\Phi(A)=\left\lbrace f\in L^0(A): \exists \lambda>0 \ I_\Phi (\lambda f)<\infty \right\rbrace.
    \]
 The following   functional 
\[\|f\|_{\Phi}=\inf\left\lbrace \lambda>0: I_\Phi\left(f/\lambda\right)\leq1\right\rbrace\]
is  the {\it Luxemburg norm} on $L^\Phi(A)$.  The space $L^\Phi(A)$ equipped with this norm is a Banach function space \cite{KamZyl3, Mus, zyl}.

Following \cite[p. 22]{Adams}, we say that a function $f\in L^1_{loc}(\Omega)$ is {\it weakly  $k$-differentiable} if for every multi-index $\alpha$ with $|\alpha|\leq k$ there exists a function $f_\alpha\in L^1_{loc}(\Omega)$ such that  for every $u\in C^\infty_C(\Omega)$ we have
\[\int\limits_\Omega f(x)\partial^\alpha u(x)dx= (-1)^{|\alpha|}\int\limits_\Omega f_\alpha(x) u(x)dx.\]
In this case we denote $\partial^\alpha f=f_\alpha$ and say that $f_\alpha$ is a {\it weak $\alpha$-derivative of $f$}.

Weak differentiability has the following useful characterization (e.g. \cite{zyl}). 
\begin{lemma}\label{characterization of weak differentiability}
Let $\Omega$ be an open subset of $\R^d$, $f,g$ be locally integrable functions on $\Omega$ and  $\alpha$ be a multi-index. The function $f$ is the $\alpha$-weak derivative of $g$ if and only if for every $u\in C_C^k(\Omega)$, where $k=|\alpha|$, we have
\[\int\limits_{\Omega} g(x)\partial^\alpha u(x)dx=(-1)^{|\alpha|}\int\limits_{\Omega}f(x)u(x)dx.\]
\end{lemma}

\begin{proof}
 Let $\Omega$ be an open subset of $\R^d$, $f,g$ be  locally integrable functions defined on $\Omega$ and $\alpha$ be a multi-index. Clearly if for every $u\in C_C^k(\Omega)$, where $k=|\alpha|$, we have
\[\int\limits_{\Omega} g(x)\partial^\alpha u(x)dx=(-1)^{|\alpha|}\int\limits_{\Omega}f(x)u(x)dx,\]
then $f$ is $\alpha$-weak derivative of $g$, as $C_C^\infty(\Omega)\subset 
C_C^k(\Omega)$.

  On the other hand, assume that $f$ is the $\alpha$-weak derivative of $g$ and take any $u\in C_C^k(\Omega)$. Recall that standard mollifier is the function $J:\R^d\to [0,\infty)$ given by the formula
\[J(x)=Ce^{\frac{-1}{1-|x|^2}}\chi_{B(0,1)}(x),\]
and that, for $r>0$ we have
$J_{(r)}(x):=\frac{1}{r^d}J\left(\frac{1}{r}x\right)$. We define the function $u_r$ by 
\[u_r(x)=(u*J_{(r)})(x),\]
where $x\in \R^d$. Notice now that since $u\in C^k_C(\Omega)$, then for small enough $r$, say for $r<R$, we have $\esssupp u_r\subset \Omega$. Clearly $u_r$ is a smooth function, since $J_{(r)}$ is smooth, hence for  $r<R$ we have  $u_r\in C_C^\infty(\Omega)$. Since $f$ is the $\alpha$-weak derivative of $g$ we get that
\[\int\limits_{\Omega} g(x)\partial^\alpha u_r(x)dx=(-1)^{|\alpha|}\int\limits_{\Omega}f(x)u_r(x)dx.\]
Moreover, by the fact that we have $\partial^\alpha u_r=J_{(r)}*\partial^\alpha u $ and that both $u$ and $\partial^\alpha u$ are continuous functions, we have that $\partial^\alpha u_r$ and $u_r$ converge uniformly to $\partial^\alpha u$ and $u$ respectively, as $r$ goes to $0$ (see \cite[Theorem 8.14, p. 242]{Folland} ). Notice now that, for $r<R/2$ the supports of all functions $u_r$ are contained in one compact set $\esssupp u_{R/2}=K\subset \Omega$. Since both $f$ and $g$ are locally integrable functions $f\chi_K$ and $g\chi_K$ generate continuous functionals on $L^{\infty}(K)$ of the form $u\mapsto \int\limits_K g(x)u(x)dx$ and  $u\mapsto \int\limits_K f(x)u(x)dx$, where $u\in L^\infty(K)$. Using the continuity of those functionals  and the fact that both $\partial^\alpha u_r$  and $u_r$ converge uniformly on $\Omega$ respectively to $\partial^\alpha u$ and $u$ as $r$ goes to $0$, we conclude that 
\begin{align*}
  \int\limits_{\Omega} g(x)\partial^\alpha u(x)dx&=\lim\limits_{r\to 0}\int\limits_{K} g(x)\partial^\alpha u_r(x)dx\\
  &=\lim\limits_{r\to 0}(-1)^{|\alpha|}\int\limits_{K}f(x)u_r(x)dx=(-1)^{|\alpha|}\int\limits_{\Omega}f(x)u(x)dx.
\end{align*}
\end{proof}

 For any  $k\in \N$, an open subset $\Omega\subset \R^d$  and  locally integrable MO function $\Phi$  on $\Omega$, define the  {\it Musielak-Orlicz-Sobolev spaces} (in short MOS spaces) \cite{Ahmida, benkir, ChGSGWK,  CUF, DHHR, Has2, Hasbook, hudzik1979, KamZyl1, KamZyl3, pepiczki, zyl}  as follows
\[W^{k,\Phi}\left(\Omega \right)=\left\{f\in L^0(\Omega):\partial^\alpha f\in L^\Phi(\Omega), \ |\alpha|\leq k \right\}.\]
We endow $W^{k,\Phi}\left(\Omega \right)$ with the norm 
\[ 
\|f\|_{W^{k,\Phi}}=\sum\limits_{|\alpha|\leq k}\|\partial^\alpha f\|_\Phi.
\]
Let $\{f_n\}_{n=1}^\infty\subset W^{k,\Phi}(\Omega)$ and $f\in W^{k,\Phi}(\Omega)$. We notice that $f_n$ converges to $f$ in  $W^{k,\Phi}(\Omega)$ if for all  $\lambda>0$ and for any multi-index $\alpha$ with $|\alpha|\leq k$,
\[\lim\limits_{n\to\infty}\int\limits_\Omega \Phi\left(x,\lambda\left|\partial^\alpha f_n(x)-\partial^\alpha f(x)\right|\right)dx=0.\]
If $\Phi(x,t) = t^{p(x)}$, where $t\ge 0$ and  $p(x)\ge 1$ for a.a. $x\in\Omega\subset  \mathbb{R}^d$, is a variable exponent function, then the corresponding Musielak-Orlicz-Sobolev space is denoted by $W^{k,p(\cdot)}(\Omega)$ and  is called variable exponent Sobolev space.

\begin{section}{Density of compactly supported smooth functions}\label{section 2}
 In essence the whole section is a step by step proof of  the main result Theorem \ref{main result of 6}. Our goal here is to establish the density of restrictions of the elements from $C_C^\infty (\R^d)$ to $\Omega$ in the space $W^{1,\Phi}(\Omega)$.

  First let us remark that analogous theorem holds in classical Sobolev spaces $W^{k,p}(\Omega)$, for any $1\leq p<\infty$, any $k\in\N$ and domain $\Omega$ whose boundary is locally a graph of a continuous function, (see for example \cite{Adams}). Unfortunately, in the case of MOS spaces the method used to prove the mentioned result cannot work. First, it uses the fact that for any $x\in \R^d$ the translation operators $\tau_x : L^p(\R^d)\to L^p(\R^d)$, given by $(\tau_x f)(y)=f(y-x)$,  $f\in L^p(\R^d)$, $y\in \R^d$ are all isometries, which is not true in a general MO space as they may be even ill-defined. Secondly the group $\{\tau_x\}_{x\in\R^d}$ is strongly continuous on $L^p(\R^d)$ in the sense that for $f\in L^p$ we have $\lim\limits_{x\to 0}\|\tau_x f-f\|_p=0$, once again this not true for $L^\Phi(\Omega)$  (see \cite[Theorem 2.1]{{kaminskacompact}}). The third problem is that in the case of $L^p(\Omega)$ one can think of it as a MO space on $\Omega$ which is given by the function $\Phi(x,t)=t^p$, where $x\in \Omega$ and $t\geq 0$. Such function can be extended to a MO function on the whole $\R^d$ in an obvious fashion, with the extension having all nice properties of the original function. In the case of a general MO function the existence of a good extension is not guaranteed.

  The reasons above suggest that a different method is required to achieve our goal. We approached the problem via the means of local extensions of both the function $\Phi$ and  elements $f\in W^{1,\Phi} (\Omega)$. Since both extensions are constructed by reflection of parts of the domain $\Omega$ about its boundary, it only works if we expect our functions $f$ to by once weakly differentiable. Moreover the method also requires higher regularity of the boundary of $\Omega$.

  We start our exposition by introducing some notation. We base our approach on that of \cite{Giovanni Leoni}. For any $x=(x_1,\dots,x_d)\in\R^d$ we define $x'\in \R^{d-1}$ to be \[x'=(x_1,..,x_{d-1}).\]
Moreover for $x\in\R^d$ such that $x=0$, we have that $x'=0$ in $\R^{d-1}$.
With a slight abuse of notation we write 
\[x=(x',x_d).\]
Recall that for $ y\in\R^d $, $d\in\N$  and $r>0$, $Q(y,r)$ is a cube with center $y$ and side length $2r$. For a fixed $d\in\N$ we will need both cubes in $\R^{d}$ and in $\R^{d-1}$. For a fixed $(y',y_d)=y\in\R^{d}$ we denote the cube in $\R^{d-1}$, centered at $y'$ and of side length $2r$ as  $Q_{d-1}(y',r)$.

  We will also need the notion of {\it rigid motion} on $\R^d$. A rigid motion on $\R^d$ is the map $T:\R^d\to\R^d$ of the form
\[T=R+c,\]
where $R$ is a rotation in $\R^d$ and $c\in\R^d$. For a fixed rigid motion $T$ and a point $x\in\R^d$ we define the {\it local coordinates of $x$} as the vector $y$, where
\[y:=T(x).\]
  Now we recall the notion of differentiable transformations between open subsets of $\R^d$. For $m\in\N$ and an open subset $U\subset \R^d$, by $C^m(U)$ we denote  all complex valued $m$-times
continuously differentiable functions defined on $U$.

\begin{definition}
 Let $U,V$ be open subsets of $\R^d$ and $\mathcal{R}:U\to V$ be a map. For a fixed $m\in\N$ we say that $\mathcal{R}$ belongs to $C^m(U,V)$, if for $i=1,\dots,d$ there exist functions $\varrho_i\in C^m(U)$ such that for $x\in U$,  
 $\mathcal{R}(x)=(\varrho_1(x),\dots,\varrho_d(x)).$
  \end{definition}
 For $\mathcal{R}\in C^m(U,V)$ and $x\in\R^d$ we define the {\it Jacobian matrix} of $\mathcal{R}$ at $x$ as
 \[ J_{\mathcal{R}}(x)=(\partial_j\varrho_i(x))_{i,j=1}^d.\]

We call $|\det J_{\mathcal{R}}(x)|$ the {\it Jacobian} of $\mathcal{R}$ at $x$.
The following substitution formula will be needed in the sequel. The proof of this result can be found in \cite[11.53, p.341]{Giovanni Leoni}.

 \begin{theorem}\label{Integration by substitution formula}(Change of variables) Let $U,V$ be open sets in $\R^d$ and $\mathcal{R}:U\to V$ be an invertible map of class $C^1(U,V)$, such that its inverse $\mathcal{R}^{-1}$ is of class $C^1(V,U)$. For any $f\in W^{1,1}(V)$ the function $f\circ\mathcal{R}$ is and element of $W^{1,1}(U)$, moreover, for any $i=1,\dots, d$ the following formula holds for a.a $x\in U$,
 \[\partial_i(f\circ \mathcal{R})(x)=\sum\limits_{j=1}^d\partial_j f(\mathcal{R}(x))\partial_i\varrho_j(x).\]
\end{theorem}
Now we can make a precise definition of a set with regular boundary.
\begin{definition}\label{boundary of class C^m}
Given an open set $\Omega\subset \R^d$ we say that $\bd(\Omega)$ is of class $C^1$ if for every $x_0\in\bd(\Omega)$ there exist an open neighborhood $U$ of $x_0$, a rigid motion $T$, a number $r>0$ and  $\mathfrak{f}\in C^1(Q_{d-1}(0,r))$ such that
\[T(x_0)=0,\]
\[T(\Omega\cap U)=\{y\in Q(0,r):y_d>\mathfrak{f}(y')\}.\]
\end{definition}

In the sequel the symbol $\mathfrak{f}$ will be restricted only to the function $\mathfrak{f}\in C^1(Q_{d-1}(0,r))$ from the above definition. 

  The next theorem will be crucial in further constructions as it provides us with both a family of very regular open neighborhoods of points in $\bd(\Omega)$ and reflection maps on those neighborhoods.

\begin{theorem}\label{reflection property}
 Let $\Omega\subset\R^d$ be open and its boundary $\bd(\Omega)$ be of class $C^1$. Then for each $x_0\in\bd(\Omega)$  there exist two triples $(V,\{V_{(t_0,t_1)}\}_{(t_0,t_1)\subset (-1,1)},\mathcal{R})$ and $(W,\{W_{(t_0,t_1)}\}_{(t_0,t_1)\subset (-1,1)},\mathcal{R'})$  where $V$ is an open neighborhood of $x_0$, $W$ is an open neighborhood of $0$, $\{V_{(t_0,t_1)}\}_{(t_0,t_1)\subset (-1,1)}$ and $\{W_{(t_0,t_1)}\}_{(t_0,t_1)\subset (-1,1)}$ are families of open sets and $\mathcal{R},\mathcal{R'}$ are respectively $C^1(V,V)$ and $C^1(W,W)$ maps.  The triple  $(W,\{W_{(t_0,t_1)}\}_{(t_0,t_1)\subset (-1,1)},\mathcal{R'})$ is of the form 
      \[W=\bigcup\limits_{t\in (-1,1)}S+At e_d,\]
and      for $-1\leq t_0<t_1\leq 1$,
    \[W_{(t_0,t_1)}=\bigcup\limits_{t\in (t_0,t_1)}S+Ate_d,\]  
    where $A>0$ is some number and 
          \[S=\{(x,\mathfrak{f}(x))\in \R^d: x\in Q_{d-1}(0,r')\},\]
where $\mathfrak{f}$ is the $C^1(Q_{d-1}(0,r))$ function associated to $x_0$ by Definition \ref{boundary of class C^m} and $0<r'<r$ is such that, for each $i=1,\dots, d-1$ we have 
      \[\sup\limits_{x\in Q_{d-1}(0,r')}|\partial_i \mathfrak{f}(x)|<\infty.\]
Furthermore the map $\mathcal{R}':W\to W$ is given by the formula
     \[\mathcal{R'}(y',\mathfrak{f}(y')+At)=(y',\mathfrak{f}(y')-At),\]
     where $(y',\mathfrak{f}(y')+At)=y\in W$. 

   The triple $(V,\{V_{(t_0,t_1)}\}_{(t_0,t_1)\subset (-1,1)},\mathcal{R})$ is connected to  $(W,\{W_{(t_0,t_1)}\}_{(t_0,t_1)\subset (-1,1)},\mathcal{R'})$  by the formulas $V=T^{-1}(W)$, $V_{(t_0,t_1)}=T^{-1}(W_{(t_0,t_1)})$ and $\mathcal{R}=T^{-1}\circ\mathcal{R'}\circ T$, where $T$ is the rigid motion associated with $x_0$ via the definition \ref{boundary of class C^m}. Moreover the triple $(V,\{V_{(t_0,t_1)}\}_{(t_0,t_1)\subset (-1,1)},\mathcal{R})$ satisfies the following conditions 
 \begin{enumerate}
     \item[$(1)$]  $V_\emptyset=\emptyset$, $V_{(-1,1)}=V$ and for each $-1\leq t_0<t_1\leq 1$ and  $-1\leq t_0'<t_1'\leq 1$, 
     \[V_{(t_0,t_1)}\cap V_{(t_0',t_1')}=V_{(t_0,t_1)\cap (t_0',t_1')}.\]
     \item[$(2)$]  \[V_{(-1,0)}=V\cap\overline{\Omega}^C\text{ and  } V_{(0,1)}=V\cap\Omega.\]
     \item[$(3)$]  If $-1\leq t_0<0<t_1\leq 1$, then
     \[\bd(\Omega)\cap V\subset V_{(t_0,t_1)}\]
     \item[$(4)$]  There exists a constant $C>0$ such that for any $(t_0,t_1)\subset (-1,1)$,
     \[|V_{(t_0,t_1)}|=C|t_1-t_0|\]
     \item[$(5)$]  $\mathcal{R}\circ\mathcal{R}=Id_V,$
     \item[$(6)$]  $\mathcal{R}_{|V\cap \bd (\Omega)}=Id_{|V\cap \bd(\Omega)},$
     \item[$(7)$]   for any $(t_0,t_1)\subset (-1,1)$,
     \[\mathcal{R}(V_{(t_0,t_1)})=V_{(-t_1,-t_0)}\]
     \item[$(8)$]  If $\mathcal{R}(x)=(\varrho_i(x),\dots,\varrho_d(x))$, then there exists a constant $C_{\mathcal R}>0$ such that, for each $i=1,\dots,d$ and each multi-index $\alpha$ with $|\alpha|=1$,
     \[\sup_{x\in V}|\partial^\alpha\varrho_i (x)|<C_{\mathcal R}.\]
 \end{enumerate}.

\end{theorem}
\begin{proof}
  Let $\Omega$ be as in the assumption. Take any $x_0\in\bd(\Omega)$ and let $r>0,U,T,\mathfrak{f}$ be as in the Definition \ref{boundary of class C^m}. We have that
  \[T(\Omega\cap U)=\{y\in Q(0,r):y_d>\mathfrak{f}(y')\}.\]
  Notice that
 \[T(\bd(\Omega)\cap U)=\{y\in Q(0,r): y_d=\mathfrak{f}(y')\}\text{ and } \mathfrak{f}(0)=0.\]
  Indeed, take any $x\in\bd(\Omega)\cap U$. Then there exists a sequence $\{x_n\}_{n=1}^\infty\subset\Omega\cap U$ such that $\lim\limits_{n\to\infty}x_n=x$. Denote $y_n=T(x_n)$, where $y_n=(y'_n,y_{n,d})$ and $T(x)=y$, then $\{y_n\}_{n=1}^\infty\subset T(\Omega\cap U)$ and $\lim\limits_{n\to \infty} y_n=y$. Hence for any $n\in\N$ we have $y_{n,d}>\mathfrak{f}(y_n')$ and $\lim\limits_{n\to\infty}y_{n,d}=y_d$. Continuity of $\mathfrak{f}$ implies that $y_d\geq \mathfrak{f}(y')$. On the other hand, suppose that $y_d>\mathfrak{f}(y')$, then $T(x)\in T(\Omega\cap U)$. Hence $x$ would be an element of $U\cap\Omega$, which cannot be true as $x\in\bd(\Omega)$. Therefore $y_d=\mathfrak{f}(y')$. In particular, since $T(x_0)=0$, we conclude that $\mathfrak{f}(0)=0$.
  
   By continuity of $\mathfrak{f}$ and $\mathfrak{f}(0)=0$ there exists $0<r'<\frac{r}{2}$ such that, for $y\in Q(0,r')\cap T(\bd(\Omega)\cap U)$ we have $|\mathfrak{f}(y')|<\frac{r}{2}$.
  Taking any $z\in\bd(Q(0,r))$ and $y\in Q(0,r')\cap T(\bd(\Omega)\cap U)$ we have $z=(z',z_d)$ and $y=(y',\mathfrak{f}(y'))$ and
  \begin{align*}
  |z-y|&=|(z',z_d)-(y',\mathfrak{f}(y'))|=|(z'-y',z_d-\mathfrak{f}(y'))|\\&\geq\max
  \{\max\limits_{i=1,\dots,d-1}|z_i-y_i|, |z_d-\mathfrak{f}(y')|\} \\
  &\geq\max
  \{\max\limits_{i=1,\dots,d-1}| \ |z_i|-|y_i| \ |, | \ |z_d|-|\mathfrak{f}(y')| \ |\}.    
  \end{align*}
  
  Since $z\in \bd(Q(0,r))$ there exists $1\leq i\leq d$ for which $|z_i|=r$. From the fact that $y\in Q_{d-1}(0,r')$ and $|y_d|=|\mathfrak{f}(y')|<\frac{r}{2}$ and $r'<\frac{r}{2}$ we have
  \begin{align}\label{distance between boundry of the set and cube}
   2dr\geq|(z',z_d)-(y',\mathfrak{f}(y'))|\geq\frac{r}{2}.   
  \end{align}

 Defining now the set $S=\{y\in Q(0,r'):y_d=\mathfrak{f}(y')\}$, we have
  \[S\subset\{y\in Q(0,r):y_d=\mathfrak{f}(y')\}=T(\bd(\Omega)\cap U).\]
Moreover, by inequality (\ref{distance between boundry of the set and cube}) we get
   \[\frac{r}{2}\leq\dist(S,\bd(Q(0,r)))\leq 2dr.\]
Let $A:=\dist(S,\bd(Q(0,r)))/(8d)$, then $A\leq\frac{r}{4}$. Define now

\[W:=\bigcup\limits_{-1<t<1}tA e_d+S\]
\[W_{(t_0,t_1)}:=\bigcup\limits_{t_0<t<t_1}tA e_d+S\]
 where $(t_0,t_1)\subset(-1,1)$ and $e_d$ is the unit  vector $(0,\dots,0,1)$ in $\R^d$.
 We will show that the sets
 \[V:=T^{-1}(W),\]
 \[V_{(t_0,t_1)}:=T^{-1}(W_{(t_0,t_1)})\]
 satisfy the conditions (1)-(4).
 Notice that, by definition of $W_{(t_0,t_1)}$, we have
 $W_\emptyset=\emptyset$, $W_{(-1,1)}=W$ and for each $-1\leq t_0<t_1\leq 1$ and  $-1\leq t_0'<t_1'\leq 1$, 
     \[W_{(t_0,t_1)}\cap W_{(t_0',t_1')}=W_{(t_0,t_1)\cap (t_0',t_1')}.\]
  So, since  $T$ a bijection, condition (1) is satisfied. 
 
  Next we show that $W\subset Q(0,r)$ and $W_{(0,1)}\subset T(\Omega\cap U)$. Take any $z\in W$,
\[z=(z',z_d)=(y',\mathfrak{f}(y')+At)\]
for some $t\in(-1,1)$ and $y\in S$. If $t=0$, then
\[z=(z',z_d)=(y',\mathfrak{f}(y'))=y\in S\subset T(\bd(\Omega)\cap U)\subset Q(0,r).\]
If $t\neq0$, then
\[z=(y',\mathfrak{f}(y')+At).\]
Since $y\in Q(0,r')$ we have that $y'$ is an element of $Q_{d-1}(0,r')$. Hence, for $i=1,\dots,d-1$,
\[|z_i|=|y_i|\leq r'<r/2.\]
We also have
\[|z_d|=|\mathfrak{f}(y')+tA|\leq|\mathfrak{f}(y')|+|t|A\leq\frac{r}{2}+A\leq\frac{r}{2}+\frac{r}{4}<r,\]
which implies that for all $i=1,\dots,d$ we have $|z_i|<r$ and so $z\in Q(0,r)$. Therefore $W\subset Q(0,r)$. Moreover, since elements of $W_{(0,1)}$ are of the form $(y',\mathfrak{f}(y')+tA)$, where $t>0$, we have that $W_{(0,1)}\subset T(\Omega\cap U)$.
 
  Notice that, since $W_{(0,1)}\subset W$ and  $W_{(0,1)}\subset T(\Omega\cap U)$  we have
 \[V_{(0,1)}=V_{(0,1)}\cap V=T^{-1}(W_{(0,1)})\cap V\subset (\Omega\cap U)\cap V={\Omega}\cap V.\]
On the other hand, if $x\in \Omega\cap V$ then there exist $(y',y_d)=y\in W$ such that $x=T^{-1}(y)$. By the definition of $W$ there exists $(z',\mathfrak{f} (z'))=z\in S$ such that $(y',y_d)=(z', \mathfrak{f}(z')+At)$ for some  $t\in (-1,1)$. Since $x\in\Omega\cap V\subset \Omega\cap U$ we conclude that $\mathfrak{f}(z')+At>\mathfrak{f}(z')$ and so $t>0$. In particular $y=(y',y_d)=(z',\mathfrak{f} (z')+At)\in W_{(0,1)}$ and so $x= T^{-1}(y)\in T^{-1}(W_{(0,1)})= V_{(0,1)}$. Therefore
\[\Omega\cap V\subset V_{(0,1)},\]
and so
\[V_{(0,1)}=V\cap \Omega.\]
Similarly, by the fact that $T$ is an isometry on $\R^d$, we have
 \[W_{(-1,0)}\subset [T(\overline{\Omega}\cap U)]^C=T((\overline{\Omega}\cap U)^C)\]
 Hence  by the fact that $V_{(-1,0)}\subset V$,
 \[V_{(-1,0)}=T^{-1}(W_{(-1,0)})\subset(\overline{\Omega}^C\cup U^C)\cap V=\overline{\Omega}^C\cap V.\]
 Similar argument as in the proof of inclusion $V\cap\Omega\subset V_{(0,1)}$ shows that any  $x\in\overline{\Omega}^C\cap V $ is of the form $x=T^{-1}(y)$ where $y\in W$. Moreover $y=(y',\mathfrak{f} (y')+At)$ for some $t< 0$, hence $x\in V_{(-1,0)} $ and so
 \[\overline{\Omega}^C\cap V\subset V_{(-1,0)}.\]
 Hence 
  \[\overline{\Omega}^C\cap V= V_{(-1,0)}.\]
We have shown that the condition (2) is satisfied.
 
  Since $S\subset T(U\cap\bd(\Omega))$, then $T^{-1}(S)\subset U\cap\bd(\Omega) $. Moreover, $T^{-1}(S)\subset V\subset U$, Hence
\[T^{-1}(S)=V\cap\bd(\Omega)\]
If $1<t_0<0<t_1<1$, then $S\subset W_{(t_0,t_1)}$ and so  $V\cap\bd(\Omega)\subset V_{(t_0,t_1)}$, hence condition (3) is satisfied.

  To show that (4) holds notice that
\[|W_{(t_0,t_1)}|=\int\limits_{Q_{d-1}(0,r')} [\mathfrak{f}(y')+At_1-(\mathfrak{f}(y')-At_0)]dy'=A(2r')^{d-1}|t_1-t_0|,\]
and since $T$ is an isometry we conclude that

\[|V_{(t_0,t_1)}|=|W_{(t_0,t_1)}|=A(2r'^{d-1})|t_1-t_0|,\]
which shows (4).

  Next we show for any $(t_0,t_1)\subset(-1,1)$, $W_{(t_0,t_1)}$ is open. Take any $z_0\in W_{(t_0,t_1)}$, then there exists $y_0\in Q(0,r')$ and $t_0<t<t_1$ such that
\[z_0=(y_0',\mathfrak{f}(y_0')+tA).\]
Take any $\varepsilon>0$ such that $At+2\varepsilon<At_1$ and $At_0<At-2\varepsilon$.
 Since $\mathfrak{f}$ is continuous there exits $s>0$ such that $Q_{d-1}(y_0',s)\subset Q_{d-1} (0,r')$ and for $y'\in Q_{d-1}(y_0',s)$, we have
 \[\mathfrak{f}(y_0')-\varepsilon<\mathfrak{f}(y')<\mathfrak{f}(y_0')+\varepsilon.\]
 Now we define 
 \[U_{z_0}=Q_{d-1}(y_0',s)\times(At-\varepsilon+\mathfrak{f}(y_0'),At+\mathfrak{f}(y_0')+\varepsilon).\]
 Then $U_{z_0}$ is an open cube in $\R^d$ and for any $z\in U_{z_0}$ we have that
 \[z=(y',z_d)\]
 for some $y'\in Q_{d-1}(y_0',s)$ and $z_d\in ( At-\varepsilon+\mathfrak{f}(y_0'),At+\mathfrak{f}(y_0')+\varepsilon)$. Notice that
 \begin{align*}
    At_0&<At-2\varepsilon<At-\varepsilon+\mathfrak{f}(y_0')-\mathfrak{f}(y')<z_d-\mathfrak{f}(y')\\&<At+\mathfrak{f}(y_0)-\mathfrak{f}(y')+\varepsilon<At+2\varepsilon<At_1.
 \end{align*}

  Let $t_{y'}:=\frac{z_d-\mathfrak{f}(y')}{A}$, then $t_{y'}\in(t_0,t_1)$ and
 $z_d=\mathfrak{f}(y')+t_{y'}A$.  Therefore  $z=(y',\mathfrak{f}(y')+t_{y'}A)\in W_{(t_0,t_1)}$ and so $U_{z_0}\subset W_{(t_0,t_1)} $, i.e. $W_{(t_0,t_1)}$ is open.
 
  Let us now define the map $\mathcal{R'}:W\to W$. For $(y',\mathfrak{f}(y')+tA)=(y',y_d)=y\in W$, we define
\begin{align*}
    \mathcal{R'}(y)=\mathcal{R'}(y',y_d)=\mathcal{R'}(y',\mathfrak{f}(y')+tA)=(y',\mathfrak{f}(y')-tA)=(y',2\mathfrak{f}(y')-y_d).
\end{align*}
Notice that for $y\in W$,
\[(\mathcal{R'}\circ\mathcal{R'})(y)=\mathcal{R'}((y',\mathfrak{f}(y')-tA))=(y',\mathfrak{f}(y')+tA)=y.\]
 Hence $\mathcal{R'}\circ\mathcal{R'}=Id_W$, and for ${(t_0,t_1)}\subset(-1,1)$ we have
 \[\mathcal{R'}(W_{(t_0,t_1)})=W_{(-t_1,-t_0)}.\]
 For $y\in W$, if we write
 \[\mathcal{R'}(y)=(\varrho'_1(y),\dots,\varrho'_d(y)),\]
 then we have 
 \[\varrho'_i(y)=y_i,\]
 for $i=1,\dots,d-1$ and for $i=d$, \[\varrho'_d(y)=2\mathfrak{f}(y')-y_d.\] Since $\mathfrak{f}$ is of class $C^1(Q_{d-1}(0,r))$, its restriction to $Q_{d-1}(0,r')$ is also of class $C^1(Q_{d-1}(0,r'))$ and for each $i=1,\dots,d$ we have
 \begin{align}\label{ograniczenie pochdnych f na kostce r'}
     \sup_{y'\in Q_{d-1}(0,r')}|\partial_i \mathfrak{f}(y)|<C_{r'}
 \end{align}
 for some $C_{r'}>0$. Moreover, 
 \begin{align}\label{partial va}
 \partial_j\varrho'_i=\delta_{ij}
 \end{align}
for $i=1,\dots,d-1$, $j=1,\dots,d-1$, 
 \[\partial_d \varrho_i'=0\]
 for $i=1,\dots,d-1$,
 \[\partial_j\varrho'_d=2\partial_j \mathfrak{f}\]
 for $j=1,\dots,d-1$, finally
 \[\partial_d\varrho_d=-1.\]
 Hence, $\mathcal{R}'$ is of class $C^1(W,W)$, and for each $y\in W$,
 \begin{align}\label{Jacobian of R'}
  |\det J_{\mathcal{R'}(y)}|=1.   
 \end{align}

   Finally we define $\mathcal{R}:V \to V$ as
\[\mathcal{R}=T^{-1} \circ \mathcal{R'} \circ T.\]
Notice that $\mathcal{R}$ together with the set $V$ and the family $V_{(t_1,t_2)}$ satisfy conditions (1)-(7) and since $T$ is a composition of a translation and a rotation, the condition (8) is also satisfied.
\end{proof}   

\begin{lemma}\label{local extension for regular domains}
    Let $(W,\{W_{(t_0,t_1)}\}_{(t_0,t_1)\subset (-1,1)},\mathcal{R'})$ be the triple from Theorem \ref{reflection property}. Let $g$ be a function in $W^{1,1}(W_{(0,1)})$. Then the function $\widetilde{g}$ defined by the formula

\[\widetilde{g}(y)=\begin{cases}g(y) & y\in W_{(0,1)}\\
0 & y\in W\setminus (W_{(-1,0)}\cup W_{(0,1)})\\
g(\mathcal{R}'(y)) & y\in W_{(-1,0)}
\end{cases}\]
is an element of $W^{1,1}(W)$ and for a.a $y=(y',y_d)\in W$ we have

\[\partial_i\widetilde{g}(y)=\begin{cases}
\partial_i g(y) & y\in W_{(0,1)}, \text{ } i=1,\dots, d \\
(\partial_i g)(\mathcal{R}'(y))+ 2(\partial_d g)(\mathcal{R}'(y))\partial_i \mathfrak{f}(y') & y\in W_{(-1,0)} \text{ } i=1,\dots, d-1 \\
-(\partial_d g)(\mathcal{R}'(y)) & y\in W_{(-1,0)}, \ i=d.
\end{cases}
\] 
\end{lemma}

\begin{proof}
  By assumption, the triple  $(W,\{W_{(t_0,t_1)}\}_{(t_0,t_1)\subset (-1,1)},\mathcal{R'})$ satisfies the conclusions of the Theorem \ref{reflection property}, in particular the sets $W_{(t_0,t_1)}$ are open for each $(t_0,t_1)\subset (-1,1)$ and of the form
    \[W_{(t_0,t_1)}=\bigcup\limits_{t\in (t_0,t_1)}S+Ate_d,\]  
    where $A>0$ is some number and 
          \[S=\{(x,\mathfrak{f}(x))\in \R^d: x\in Q_{d-1}(0,r')\},\]
where $\mathfrak{f}$ is of class $C^1(Q_{d-1}(0,r))$ such that, for some $0<r'<r$ and any $i=1,\dots, d-1$, \[\sup\limits_{x\in Q_{d-1}{0,r')}} |\partial_i\mathfrak{f}(x)|<\infty.\]

   Take any $0<\varepsilon<1/3$, and for $(y',y_d)=y\in W$ define

 \[F_{\varepsilon,1}(y)=\int\limits_{-\infty}^{y_d} \frac{15}{16A^5\varepsilon^5}(\mathfrak{f} (y')+A\varepsilon-t)^2(\mathfrak{f}(y')+3A\varepsilon-t)^2\chi_{[\mathfrak{f}(y')+A\varepsilon, \mathfrak{f}(y')+3A\varepsilon]}(t)dt,\]

\[F_{-1,-\varepsilon}(y)=\int\limits_{y_d}^\infty\frac{15}{16A^5\varepsilon^5}(\mathfrak{f}(y')-A\varepsilon-t)^2(\mathfrak{f}(y')-3A\varepsilon-t)^2\chi_{[\mathfrak{f}(y')-3A\varepsilon, \mathfrak{f}(y')-A\varepsilon]}(t)dt,\]
   \[F_{-\varepsilon,\
 \varepsilon}(y)= \chi_W(y)- F_{\varepsilon,1}(y)-F_{-1,-\varepsilon}(y). \] 
Clearly, $F_{\varepsilon,1}$ and $F_{-1,-\varepsilon}$ are non-negative functions. Moreover, by the Fundamental Theorem of Calculus, for any $(y',y_d)=y\in W$  the function $t\mapsto F_{\varepsilon,1}(y',t)$ is non-decreasing and $t\mapsto F_{-1,-\varepsilon}(y',t)$ is non-increasing function, where $t\in (-A,A)$. Moreover we have
\begin{align}\label{range of F,1,varepsilon,1}
    F_{\varepsilon,1}: W\to [0,1].
\end{align}
 Indeed,
take any $(y',y_d)=y\in W_{(\varepsilon,1)}$. By definition of $W_{(\varepsilon,1)}$ we have that $y_d< \mathfrak{f}(y')+A\varepsilon $, and so
\[F_{\varepsilon,1}(y)=\int\limits_{-\infty}^{y_d} \frac{15}{16A^5\varepsilon^5}(\mathfrak{f}(y')+A\varepsilon-t)^2(\mathfrak{f}(y')+3A\varepsilon-t)^2\chi_{[\mathfrak{f}(y')+A\varepsilon, \mathfrak{f}(y')+3A\varepsilon]}(t)dt=0.\]
For $(y', y_d)=y\in W$ such that $\mathfrak{f}(y')+A\varepsilon\leq y_d\leq \mathfrak{f}(y')+3A\varepsilon$, i.e. $y\in W\setminus \left( W_{(-1,\varepsilon)}\cup W_{(3\varepsilon,1)}\right)$, we have that
\begin{align*}
   & F_{\varepsilon,1}(y) =\int\limits_{-\infty}^{y_d} \frac{15}{16A^5\varepsilon^5}(\mathfrak{f}(y')+A\varepsilon-t)^2(\mathfrak{f}(y')+3A\varepsilon-t)^2\chi_{[\mathfrak{f}(y')+A\varepsilon, \mathfrak{f}(y')+3A\varepsilon]}(t)dt
   \\&=\int\limits_{\mathfrak{f}(y')+A\varepsilon}^{y_d}\frac{15}{16A^5\varepsilon^5}(\mathfrak{f}(y')+A\varepsilon-t)^2(\mathfrak{f}(y')+3A\varepsilon-t)^2dt 
  \end{align*} 
 By substitution $u=\mathfrak{f}(y')+A\varepsilon-t$,
 \begin{align*}
   &F_{\varepsilon,1}(t)=-\frac{15}{16 A^5\varepsilon^5}\int\limits_{0}^{\mathfrak{f}(y')+A\varepsilon-y_d}t^2(t+2A\varepsilon)^2dt=\frac{15}{16 A^5\varepsilon^5}\int\limits_{\mathfrak{f}(y')+A\varepsilon-y_d}^0t^2(t+2A\varepsilon)^2dt
   \\&=\frac{15}{16 A^5\varepsilon^5}\left(\frac{(y_d-\mathfrak{f}(y')-A\varepsilon)^5}{5}-\frac{4A\varepsilon(\mathfrak{f}(y')+A\varepsilon-y_d)^4}{4}-\frac{4A^2\varepsilon^2(\mathfrak{f}(y')+A\varepsilon-y_d)^3}{3}\right)
   \\&=\frac{-\left(\mathfrak{f}(y')+A\varepsilon-y_d \right)^3\left(3(\mathfrak{f}(y')+A\varepsilon-y_d)^2+15A\varepsilon(\mathfrak{f}(y')+A\varepsilon-y_d)+20A^2\varepsilon^2\right)}{16A^5\varepsilon^5}.
\end{align*}
Notice that also, for $y=(y',\mathfrak{f}(y+3A\varepsilon)\in W$ we have
\[F_{\varepsilon,1}(y)=\frac{-1}{16A^5\varepsilon^5}(-2A\varepsilon)^3(3(-2A\varepsilon)^2+15A\varepsilon(-2A\varepsilon)+20A^2\varepsilon^2)=1.\]
Now, if $(y',y_d)\in W$ and $y_d>\mathfrak{f}(y')+3A\varepsilon$ we have that
\begin{align*}
    F_{\varepsilon,1}(y) &=\int\limits_{-\infty}^{y_d} \frac{15}{16A^5\varepsilon^5}(\mathfrak{f}(y')+A\varepsilon-t)^2(\mathfrak{f}(y')+3A\varepsilon-t)^2\chi_{[\mathfrak{f}(y')+A\varepsilon, \mathfrak{f}(y')+3A\varepsilon]}(t)dt= \\&=\int\limits_{\mathfrak{f}(y')+A\varepsilon}^{\mathfrak{f}(y')+3A\varepsilon}\frac{15}{16A^5\varepsilon^5}(\mathfrak{f}(y')+A\varepsilon-t)^2(\mathfrak{f}(y')+3A\varepsilon-t)^2dt= 1.
\end{align*}
And so the function $F_{\varepsilon,1}$ is of the form
\begin{align}\label{wzor na F,1,varepsilon}
 F_{\varepsilon,1}(y)=\begin{cases}
    0 &  y\in W_{(-1,\varepsilon)} \\
    \frac{\left(y_d-\mathfrak{f}(y')-A\varepsilon \right)^3\left(3(\mathfrak{f}(y')+A\varepsilon-y_d)^2+15A\varepsilon(\mathfrak{f}(y')+A\varepsilon-y_d)+20A^2\varepsilon^2\right)}{16A^5\varepsilon^5} & y\in W\cap \overline{W_{(\varepsilon,3\varepsilon)}} \\
    1 & y\in W_{(3\varepsilon,1)}.
 \end{cases}   
\end{align}
Now we can also see that
\begin{align}\label{support F,1,varepsilon}
    \esssupp F_{\varepsilon,1}=\overline{ W_{(\varepsilon,1)}}.
\end{align}
 Indeed, if $y\in \overline{W_{(\varepsilon,1)}}$, then there exists a sequence $\{y_k\}_{k=1}^\infty\subset W_{(\varepsilon,1)}$ such that $\lim\limits_{k\to \infty} y_k=y$. By the formula (\ref{wzor na F,1,varepsilon}) we have that, for each $k\in \N$, $F_{\varepsilon,1}(y_k)\neq 0$, hence $y_k\in \esssupp F_{\varepsilon,1}$. But $\esssupp F_{1,\varepsilon}$ is a closed set and so $\overline{ W_{(\varepsilon,1)}}\subset \esssupp F_{1,\varepsilon}$. Similarly, if $y\in \esssup F_{\varepsilon,1}$, there exists a sequence of points $\{y_k\}_{k=1}^\infty$, such that, for each $k\in\N$ we have $F_{\varepsilon,1}(y_k)> 0$ and $\lim\limits_{k\to \infty} y_k=y$. Inspection of the formula (\ref{wzor na F,1,varepsilon}) shows that $F_{\varepsilon,1}(z)\neq 0$ only for $z \in W_{(\varepsilon,1)}$. Therefore, we conclude that $\{y_k\}_{k=1}^\infty\subset W_{(1,\varepsilon)}$ and so we conclude that $y_k\in \overline{W_{(0,1)}}$.

  Similar reasoning to the one above  shows that for $y\in W$ we have
\begin{align}\label{wzor na F,-1,-varepsilon}
    F_{-1,-\varepsilon}(y)=\begin{cases}
       1 & y\in W_{(-1,-3\varepsilon)} \\
    \frac{\left(\mathfrak{f}(y')-A\varepsilon-y_d \right)^3\left(2(\mathfrak{f}(y')+A\varepsilon+y_d)^2+15A\varepsilon(y_d-\mathfrak{f}(y')+A\varepsilon)+20A^2\varepsilon^2\right)}{16A^5\varepsilon^5} & y\in W\cap \overline{W_{(-3\varepsilon,-\varepsilon)}} \\
    0  & y\in W_{(-\varepsilon,1)},
    \end{cases}
\end{align}
and
\begin{align}\label{range F,-1,-varepslion}
 \  F_{-1,-\varepsilon}: W\to [0,1],
\end{align}
\begin{align}\label{support F,-1,-varepsion}
    \esssupp F_{-1,-\varepsilon}=\overline{W_{(-1,-\varepsilon)}}.
\end{align}
Since $F_{-\varepsilon,\varepsilon}=\chi_W-F_{\varepsilon,1}- F_{-1,-\varepsilon}$, by the fact that supports of $F_{\varepsilon,1}$ and $F_{-1,-\varepsilon}$ are disjoint (see (\ref{support F,1,varepsilon}) and (\ref{support F,-1,-varepsion}))  and the fact that ranges of both $F_{\varepsilon,1}$ and $F_{-1,-\varepsilon}$ are equal to $[0,1]$ (see (\ref{range of F,1,varepsilon,1}) and (\ref{range F,-1,-varepslion})) we have that
\begin{align}\label{range of F,-varepsilon,varepsilon}
F_{-\varepsilon,\varepsilon}:W\to [0,1].    
\end{align}
Examining the formulas (\ref{wzor na F,1,varepsilon}) and (\ref{wzor na F,-1,-varepsilon}) we see that $F_{-\varepsilon,\varepsilon}(z)\neq 0$ only for $z\in W_{(-3\varepsilon,3\varepsilon)}$. Hence, reasoning as in the case of $F_{\varepsilon,1}$ we deduce that 
 \begin{align*}
     \esssupp F_{-\varepsilon,\varepsilon}= \overline{W_{(-3\varepsilon,3\varepsilon)}}.
 \end{align*}
 Moreover we also see that,
 for $y\in W_{(-\varepsilon,\varepsilon)}$, we have
 \begin{align*}
     F_{-\varepsilon,\varepsilon}(y)=1.
 \end{align*}
 Notice that, by the Fundamental Theorem of Calculus and definitions of $F_{\varepsilon,1}$, $F_{-1,-\varepsilon}$, for $(y,y_d)=y\in W$ we have
 \begin{align}\label{d partial of F_1}
\partial_d F_{\varepsilon,1}(y)=\frac{15}{16A^5\varepsilon^5}(\mathfrak{f}(y')+A\varepsilon -y_d)^2(\mathfrak{f}(y')+3A\varepsilon-y_d)^2\chi_{[\mathfrak{f}(y')+A\varepsilon, \mathfrak{f}(y')+3A\varepsilon]}(y_d),
\end{align}

and 
 \begin{align*}
\partial_d F_{-1,-\varepsilon}(y)=\frac{-15}{16A^5\varepsilon^5}(\mathfrak{f}(y')-A\varepsilon -y_d)^2(\mathfrak{f}(y')-3A\varepsilon-y_d)^2\chi_{[\mathfrak{f}(y')-3A\varepsilon, \mathfrak{f}(y')-A\varepsilon]}(y_d).
\end{align*}  
Take any $y\in W$ and if we set $a= \mathfrak{f}(y')$, $b=A\varepsilon$ and $x=y_d$, then \[\partial_d F_{\varepsilon,1}(y)=\frac{15}{16b^5}(a+b-x)^2(a+3b-x)^2\chi_{[a+b,a+3b]}(x).\]
Hence, for a fixed $a$ and $b$ the supremum over $x$ of the right side of the equation above is attained at $x=a+2b$ and is equal to $\frac{15}{b}$. And so, for a fixed $y'\in Q_{d-1}(0,r')$, we have that
\[\sup\limits_{\mathfrak{f}(y')-A<y_d<\mathfrak{f}(y')+A}\partial_d F_{\varepsilon,1}(y',y_d)=\frac{15}{A\varepsilon}.\]
Consequently,
\begin{align}\label{bound for the derivastive of F}
    \sup\limits_{y\in W}\partial_d F_{\varepsilon,1}(y)=\sup\limits_{y'\in Q_{d-1}(0,r')}\left(\sup\limits_{\mathfrak{f}(y')-A<y_d<\mathfrak{f}(y')+A}\partial_d F_{\varepsilon,1}(y',y_d)\right)=\frac{15}{A\varepsilon}.
\end{align}
Now for any $(y',t)\in Q_{d-1}(0,r')\times \R$ define
\begin{align*}
 K_{\varepsilon,1}(y',t)=\frac{15}{16A^5\varepsilon^5}(\mathfrak{f}(y')+A\varepsilon -t)^2(\mathfrak{f}(y')+3A\varepsilon-t)^2\chi_{[\mathfrak{f}(y')+A\varepsilon, \mathfrak{f}(y')+3A\varepsilon]}(t),   
\end{align*}
and notice that for $t\in \R$ and $y'\in Q_{d-1}(0,r')$, we have
\begin{align*}
\chi_{[\mathfrak{f}(y')+A\varepsilon, \mathfrak{f}(y')+3A\varepsilon]}(t)=\chi_{[t-3A\varepsilon, t-A\varepsilon]}(\mathfrak{f}(y')).
\end{align*}
Fix any $t\in \R$ and take any $y'\in Q_{d-1}(0,r')$. First assume that $\mathfrak{f}(y')\in (-\infty, t-3A\varepsilon)\cup(t-A\varepsilon,\infty)$, then for $i=1,\dots, d-1$ we have
\[\partial_i K_{\varepsilon,1}(y',t)=0.\]
If now $\mathfrak{f}(y')\in(t-3A\varepsilon,t-A\varepsilon)$, then
\begin{align*}
 \partial_i K_{\varepsilon,1}(y',t)&=\frac{15}{4A^5\varepsilon^5}\partial_i \mathfrak{f}(y')(\mathfrak{f}(y')+A\varepsilon -t)(\mathfrak{f}(y')+3A\varepsilon-t)(\mathfrak{f}(y')+2A\varepsilon-t)\\&\cdot\chi_{[\mathfrak{f}(y')+A\varepsilon, \mathfrak{f}(y')+3A\varepsilon]}(t),    \end{align*}
Now consider the case $\mathfrak{f}(y')=t-3A\varepsilon$. Let $i=1,\dots,d-1$ and take any $h\in \R$ such that $y'+he_i\in Q_{d-1}(0,r')$ we have
\begin{align*}
    \frac{K_{\varepsilon,1}(y'+he_i,t)-K_{\varepsilon,1}(y',t)}{h}&=\frac{K_{\varepsilon,1}(y'+he_i,\mathfrak{f}(y')+3A\varepsilon)-K_{\varepsilon,1}(y',\mathfrak{f}(y')+3A\varepsilon)}{h}\\&=
    \frac{K_{\varepsilon,1}(y'+he_i,\mathfrak{f}(y')+3A\varepsilon)}{h}\\&=
    \frac{15}{16A^5\varepsilon^5h}(\mathfrak{f}(y'+he_i)-\mathfrak{f}(y')-2A\varepsilon)^2\\&\cdot(\mathfrak{f}(y'+he_i)-\mathfrak{f}(y'))^2\chi_{[\mathfrak{f}(y'), \mathfrak{f}(y')+2A\varepsilon]}(\mathfrak{f}(y'+he_i)).
\end{align*}
By the Mean Value Theorem, for some $\theta\in [0,1]$ we have
\begin{align*}
   \frac{K_{\varepsilon,1}(y'+he_i,t)-K_{\varepsilon,1}(y',t)}{h}&=
   \frac{15}{16A^5\varepsilon^5h}(\partial_i \mathfrak{f}(y'+\theta he_i)h-2A\varepsilon)^2(\partial_i \mathfrak{f}(y'+\theta he_i)h)^2\\&\cdot\chi_{[\mathfrak{f}(y'), \mathfrak{f}(y')+2A\varepsilon]}(\mathfrak{f}(y')+\partial_i\mathfrak{f}(y'+\theta he_i)h).
\end{align*}
Thus, for small enough $h$, if $\partial_i\mathfrak{f}(y'+\theta he_i)h>0$ we have that
\[\frac{K_{\varepsilon,1}(y'+he_i,t)-K_{\varepsilon,1}(y',t)}{h}= \frac{15}{16A^5\varepsilon^5}(\partial_i \mathfrak{f}(y'+\theta he_i)h-2A\varepsilon)^2(\partial_i \mathfrak{f}(y'+\theta he_i)h)^2\]
and if $\partial_i\mathfrak{f}(y'+\theta he_i)h\leq 0$ we have
\[\frac{K_{\varepsilon,1}(y'+he_i,t)-K_{\varepsilon,1}(y',t)}{h}=0.\]
Finally, for small enough $h$, we have
\begin{align*}
    \left| \frac{K_{\varepsilon,1}(y'+he_i,t)-K_{\varepsilon,1}(y',t)}{h}\right|\leq \left|\frac{15}{16A^5\varepsilon^5}(\partial_i \mathfrak{f}(y'+\theta he_i)h-2A\varepsilon)^2(\partial_i \mathfrak{f}(y'+\theta he_i)h)^2 \right|.
\end{align*}
 By the assumption  $\sup\limits_{x\in Q_{d-1}(0,r')}|\partial_i \mathfrak{f}(x)|<\infty $, we have for $i=1,\dots,d-1$,
\[\partial_i K_{\varepsilon,1}(y',t)=\lim\limits_{h\to 0}\frac{K_{\varepsilon,1}(y'+he_i,t)-K_{\varepsilon,1}(y',t)}{h}=0.\]
Similarly, if $\mathfrak{f}(y')=t-A\varepsilon$ we have
\[\partial_i K_{\varepsilon,1}(y',t)=0.\] 
From this we conclude that the function $K_{\varepsilon,1}$ has continuous partial derivatives, for $i=1,\dots,d-1$, on $Q_{d-1}(0,r')\times \R$, given by the formula
\begin{align*}
    \partial_i K_{\varepsilon,1}(y',t)&=\frac{15}{4A^5\varepsilon^5}\partial_i \mathfrak{f}(y')(\mathfrak{f}(y')+A\varepsilon -t)(\mathfrak{f}(y')+3A\varepsilon-t)(\mathfrak{f}(y')+2A\varepsilon-t)\\&\cdot\chi_{[\mathfrak{f}(y')+A\varepsilon, \mathfrak{f}(y')+3A\varepsilon]}(t).
\end{align*}
 Hence, by the Leibniz Integral Rule \cite[Theorem 2.27, p.56]{Folland}  we have that
 \begin{align*}
     \partial_i F_{\varepsilon,1}(y)&=\int\limits_{-\infty}^{y_d}\partial_i K_{\varepsilon,1}(y',t)dt=
    \int\limits_{-\infty}^{y_d}\frac{15}{4A^5\varepsilon^5}\partial_i \mathfrak{f}(y')(\mathfrak{f}(y')+A\varepsilon -t)(\mathfrak{f}(y')+3A\varepsilon-t)\\&\cdot (\mathfrak{f}(y')+2A\varepsilon-t)\chi_{[\mathfrak{f}(y')+A\varepsilon, \mathfrak{f}(y')+3A\varepsilon]}(t)dt\\&=
    \frac{-15}{16A^5\varepsilon^5}\partial_i \mathfrak{f}(y')(\mathfrak{f}(y')+A\varepsilon -y_d)^2(\mathfrak{f}(y')+3A\varepsilon-y_d)^2\chi_{[\mathfrak{f}(y')+A\varepsilon, \mathfrak{f}(y')+3A\varepsilon]}(y_d)\\&=-\partial_{i}\mathfrak{f}(y')K_{\varepsilon,1}(y),
 \end{align*}
 and so, by (\ref{d partial of F_1}), for $i=1,\dots,d-1$,
 \begin{align}\label{1.21}
  \partial_i F_{\varepsilon,1}(y)=   -\partial_i \mathfrak{f}(y')\partial_d F_{\varepsilon,1}(y).
\end{align}
Arguing similarly, we can show that
  \begin{align}\label{costam}
\partial_i F_{-1,-\varepsilon}(y)=-\partial_i \mathfrak{f}(y')\partial_d F_{-1,-\varepsilon}(y).
 \end{align}
 Since any $y\in W$ can be written as $y=(y',y_d)=(y',\mathfrak{f}(y')+At)$ for some $t\in (-1,1)$, and since 
 \[\mathcal{R}'(y)=(y',\mathfrak{f}(y')-At)=(y',2\mathfrak{f}(y')-y_d),\] 
we have that
\begin{align*}
    \partial_d F_{\varepsilon,1}(\mathcal{R}'(y))&=\frac{15}{16A^5\varepsilon^5}(-\mathfrak{f}(y')+A\varepsilon+y_d)^2(-\mathfrak{f}(y')+3A\varepsilon+y_d)^2\\&\cdot\chi_{[\mathfrak{f}(y')+A\varepsilon, \mathfrak{f}(y')+3A\varepsilon]}(2\mathfrak{f}(y')-y_d))\\&
    =\frac{15}{16A^5\varepsilon^5}(\mathfrak{f}(y')-A\varepsilon -y_d)^2(\mathfrak{f}(y')-3A\varepsilon-y_d)^2\chi_{[\mathfrak{f}(y')-3A\varepsilon, \mathfrak{f}(y')-A\varepsilon]}(y_d)\\&=-\partial_d F_{-1,-\varepsilon}(y).
\end{align*}
We conclude that for any $y\in W$,
\begin{align}\label{rownosc w odbiciu dla d}
    \partial_d F_{\varepsilon,1}(\mathcal{R}'(y))=-\partial_d F_{-1,-\varepsilon}(y),
\end{align}
and for $i=1,\dots d-1$, by (\ref{1.21}), (\ref{costam}), (\ref{rownosc w odbiciu dla d}) we have

\begin{align}\label{rownosc w odbiciu dla i}
     \partial_i F_{\varepsilon,1}(\mathcal{R}'(y))=-\partial_i \mathfrak{f}(y')\partial_d F_{\varepsilon,1}(\mathcal{R'}(y))=\partial_i \mathfrak{f}(y')\partial_d F_{-1,-\varepsilon}(y)=  - \partial_i F_{-1,-\varepsilon}(y).
\end{align}
From the above equalities, the fact that $\mathcal{R'}(W_{(\varepsilon,1)})=W_{(-1,-\varepsilon)}$, and by (\ref{bound for the derivastive of F}) we conclude that
  \begin{align}\label{d partial of F}
    \sup\limits_{y\in W} |\partial_d F_{\varepsilon,1}(y)|=\sup\limits_{y\in W}  |\partial_d F_{-1,-\varepsilon}(y)|=\frac{15}{A\varepsilon}.
\end{align}
and for $i=1,\dots, d-1$, by the assumption $\sup\limits_{x\in Q_{d-1}(0,r')}|\partial_i \mathfrak{f}(x)|<\infty $ and (\ref{1.21}), the fact that any $y\in W$ can be written as $y=\mathcal{R}'(\mathcal{R}'(y))$ and (\ref{rownosc w odbiciu dla i}) we have
 \begin{align}\label{i partial of F}
   \sup\limits_{y\in W}  |\partial_i F_{-1,-\varepsilon}(y)|= \sup\limits_{y\in W} |\partial_i F_{\varepsilon,1}(y)|=\sup\limits_{y\in W}|\partial_i\mathfrak{f}(y')||\partial_d F_{\varepsilon,1}(y)| \leq\frac{15}{A\varepsilon}\sup_{x\in Q_{d-1}(0,r')}|\partial_i \mathfrak{f}(x)|<\infty.
 \end{align}
We use the function $F_{\varepsilon,1}$, $F_{-\varepsilon,\varepsilon}$ and $F_{-1,-\varepsilon}$ to partition any function form $C_C^1(W)$ into $C_C^1(W)$ functions with supports that will be placed conveniently placed in $W$. 
Take any $u\in C^1_C(W)$ and $i=1,\dots, d$ and $0<\varepsilon<\frac{1}{3}$ and define the functions 
\begin{align}\label{definicja u1}
  u_{1,\varepsilon}= F_{\varepsilon,1}u,  
\end{align}
\begin{align}\label{definicja u2}
    u_{2,\varepsilon}= F_{-\varepsilon,\varepsilon}u,
\end{align}
\begin{align}\label{definicja u3}
    u_{3,\varepsilon}=F_{-1,-\varepsilon}u.
\end{align}

  By definitions of functions $F_{\varepsilon,1}$, $F_{-\varepsilon,\varepsilon}$, $F_{-1,-\varepsilon}$, we conclude that 
\[u_{1,\varepsilon}+u_{2,\varepsilon}+u_{3,\varepsilon}=u.\]
Notice now, that $u_{1,\varepsilon}(y)\neq0$ only if both $u(y)\neq 0$ and $F_{\varepsilon,1}(y)\neq 0$, therefore if $u_{1,\varepsilon}(y)\neq 0$, by formula (\ref{wzor na F,1,varepsilon}), we have that $y\in W_{(1,\varepsilon)}$. Hence $\esssupp u_{1,\varepsilon}\subset \overline{W_{(\varepsilon,1)}}\cap\esssupp u$. And since $\esssupp u$ is a compact subset of $W$, we have that $\esssupp u\cap \bd (W)$ is empty and so 
\begin{align}\label{support u_1}
   \esssupp u_{1,\varepsilon}\subset \overline{W_{(\varepsilon,1)}}\setminus (\bd(W)\cap\esssupp u). 
\end{align}

Recall that
\[W_{(\varepsilon,1)}=\bigcup\limits_{t\in (\varepsilon,1)} S+Ate_d, \]
where $S$ is the graph of $\mathfrak{f}$, 
\[S=\{(y',y_d)\in \R^d: y'\in Q_{d-1}(0,r'), y_d=\mathfrak{f}(y')\}.\]
Since $\mathfrak{f}$ is a bounded continuous function on $Q_{d-1}(0,r')$ it can be uniquely extended to a continuous function $\overline{\mathfrak{f}}$ on $\overline{Q_{d-1}(0,r')}$, such that $\overline{\mathfrak{f}}(x)=\mathfrak{f}(x)$ for $x\in Q_{d-1}(0,r')$. If $\{y_k\}_{k=1}^\infty=\{(y'_k,\mathfrak{f}(y'_k))\}_{k=1}^\infty$ is a sequence of points from $S$ converging to some point $(y',y_d)$, then $y'\in \overline{Q_{d-1}(0,r')}$ and $y_d=\lim\limits_{k\to\infty} \mathfrak{f}(y_k')=\lim\limits_{k\to\infty} \overline{\mathfrak{f}}(y_k')=\overline{\mathfrak{f}}(y')$. Therefore 
\begin{align*}
\overline{S}\subset\{(y',y_d)\in \R^d: \ y'\in \overline{Q_{d-1}(0,r')}, \ y_d=\overline{\mathfrak{f}}(y') \}    
\end{align*}

Similarly, for any point $(y',y_d)=y\in \{(y',y_d)\in \R^d: \ y'\in \overline{Q_{d-1}(0,r')}, \ y_d=\overline{\mathfrak{f}}(y') \}$ we have $y_d=\overline{\mathfrak{f}}(y')$. Now we show that $(y',y_d)\in \overline{S}$. Since $y'$ is an element of $\overline{Q_{d-1}(0,r')}$ there exists a sequence $\{x_k\}_{k=1}^\infty$  in $Q_{d-1}(0,r')$ convergent to $y'$ in $\R^{d-1}$. By continuity of $\mathfrak{f}$  the sequence $\{(x_k, \mathfrak{f}(x_k))\}_{k=1}^\infty\subset S$ convergent to $(y',\overline{\mathfrak{f}}(y'))=(y',y_d)$ in $\R^d$. Hence
\[\overline{S}=\{(y',y_d)\in \R^d: \ y'\in \overline{Q_{d-1}(0,r')}, \ y_d=\overline{\mathfrak{f}}(y') \}.\]
Now we will show that, for any $(a,b)\subset (-1,1)$ we have that
\begin{align}\label{clousure of Wab}
    \overline{W_{(a,b)}}=\bigcup\limits_{t\in[a,b]} \overline{S}+Ate_d.
\end{align}

If $(y',y_d)=y$ is an element of $\overline{W_{(a,b)}}$, then there exists a sequence $\{y_k\}_{k=1}^\infty=\{(y_k',y_{k,d})\}_{k=1}^\infty$ of points in $W_{(a,b)}$ such that $\lim\limits_{k\to\infty}y_k =y$. For each $k\in\N$, $y_k$ is an element of $W_{(a,b)}$ and so it can be written as $(y_k',\mathfrak{f}(y_k')+At_k)$, where $t_k\in (a,b)$ and $y_k'\in Q_{d-1}(0,r')$. Since the sequence $\{y_k\}_{k=1}^\infty$ is convergent so are the sequences $\{y'_k\}_{k=1}^\infty$ and $\{\mathfrak{f}(y'_k)+At_k\}$ are convergent to $y'$ in $\R^{d-1}$ and  to $y_d$ in $\R$, respectively. Since each $y'_k$ is an element of $Q_{d-1}(0,r')$ we have that $y'\in \overline{Q_{d-1}(0,r')}$. Moreover, since for $x\in Q_{d-1}(0,r')$ we have $\overline{\mathfrak{f}}(x)=\mathfrak{f}(x)$, we can write

\[y_d=\lim\limits_{k\to\infty}(\mathfrak{f}(y_k')+At_k)=\lim\limits_{k\to\infty}(\overline{\mathfrak{f}}(y_k')+At_k).\]

By continuity of $\overline{\mathfrak{f}}$ we conclude that $\lim\limits_{k\to\infty}\overline{\mathfrak{f}}(y_k')=\overline{\mathfrak{f}}(y')$ and so the sequence $\{t_k\}_{k=1}^\infty$ is also convergent to some $t\in [a,b]$. Hence 
\[y\in \{(y',y_d)\in \R^d: y'\in \overline{Q_{d-1}(0,r')}, \ y_d= \overline{\mathfrak{f}}(y')+At, \ t\in [a,b]\}.\]
Therefore
\[\overline{W_{(a,b)}}\subset \bigcup_{t\in [a,b]}\left\{(y',y_d)\in\R^d: y'\in \overline{Q_{d-1}(0,r')}, \ y_d=\overline{\mathfrak{f}}(y')\right\}+Ate_d=\bigcup\limits_{t\in[a,b]} \overline{S}+Ate_d.\]
If now $(y',y_d)$ is an element of $\bigcup\limits_{t\in[a,b]} \overline{S}+Ate_d.$ it can be written as $(y',\overline{\mathfrak{f}}(y')+At)$, where $y'\in\overline{Q_{d-1}(0,r')}$ and $t\in [a,b]$. Then there exists a sequence of points $\{x_k\}_{k=1}^\infty\in Q_{d-1}(0,r')$ such that $\lim\limits_{k\to\infty}x_k=y'$ in $\R^{d-1}$ and a sequence of numbers $\{t_k\}\subset (a,b)$ convergent to $t$. Therefore the sequence
$\{(x_k,\mathfrak{f}(x_k)+At_k)\}_{k=1}^\infty\subset W_{(a,b)}$ converges to $(y',y_d)$ in $\R^d$
and so $\bigcup\limits_{t\in[a,b]} \overline{S}+Ate_d\subset \overline{W_{(a,b)}}$.
In particular, we have (\ref{clousure of Wab}).

  For $i=1,\dots,d-1$ and $(a,b)\subset (-1,1)$, we define 
\begin{align*}
 \text{Face}_i(W_{(a,b)})&=\bigcup_{t\in [a,b]}\left\{(y',y_d)=(y_1,\dots,y_i,\dots,\mathfrak{f}(y')+At)\in \overline{Q_{d-1}(0,r')}: \  |y_i|=r'\right\}.    
\end{align*}
Then we can write 
\begin{align*}
  \bd(W_{(a,b)})&=\overline{W_{(a,b)}}\setminus W_{(a,b)}=\left(\bigcup\limits_{t\in[a,b]} \overline{S}+Ate_d\right)\setminus \left(\bigcup\limits_{t\in (a,b)} S+Ate_d\right)\\&=(\overline{S}+ae_d)\cup(\overline{S}+be_d)\cup\bigcup\limits_{i=1}^{d-1} \text{Face}_i(W_{(a,b)}).
\end{align*}
Hence, we get that for any $0<\varepsilon<1$, we have that
\begin{align*}
    \overline{W_{(\varepsilon,1)}}\setminus \bd (W)= W_{(\varepsilon,1)}\cup (S+A\varepsilon e_d)=\bigcup\limits_{t\in[\varepsilon,1)} S+Ate_d.
\end{align*}
and so, by (\ref{support u_1}) we have that
\begin{align}\label{nosnik u1e}
 \esssupp u_{1,\varepsilon}&\subset (\overline{W_{(1,\varepsilon)}}\setminus \bd(W))\cap\esssupp u\subset\bigcup\limits_{t\in[\varepsilon,1)} S+Ate_d\\&\notag\subset \bigcup\limits_{t\in(\varepsilon/2,1)} S+Ate_d= W_{(\varepsilon/2,1)}.   
\end{align}

Hence 

\begin{align}\label{nosniki u_1}
    u_{1,\varepsilon}\in C_C^1(W_{(\varepsilon/2,1)}).
\end{align}
Similarly, one can show that
\begin{align}\label{nosniki u_2 i u_3}
  u_{3,\varepsilon}\in C_C^1(W_{(-1,-\varepsilon/2)}), \ \   \esssupp u_{2,\varepsilon}\subset \overline{W_{(-3\varepsilon,3\varepsilon})}.
\end{align}

  Take now any $y\in W_{(0,1)}$. By the definition of $W_{(0,1)}$, there exists $t_y\in(0,1)$ such that $y\in S+At_ye_d$, and so if $3\varepsilon<t_y$ we have that $y\in W_{(3\varepsilon,1)}$. Hence for $0<\varepsilon<t_y/3$, by formula (\ref{wzor na F,1,varepsilon}), 
\[F_{1,\varepsilon}(y)=1.\]
It follows in view of (\ref{nosnik u1e}),
\begin{align*}
\lim\limits_{\varepsilon\to 0^+}u_{1,\varepsilon}=u\chi_{W_{(0,1)}},
\end{align*}
where the above limit is taken pointwise. Similar reasoning shows that 
\begin{align*}
  \lim\limits_{\varepsilon\to 0^+}u_{3,\varepsilon}=u\chi_{W_{(-1,0)}}.  
\end{align*}

   Take now a function $g$, such that $g\in W^{1,1}(W_{(0,1)})$. We will show, that the function $\widetilde{g}$ defined by the formula
\[\widetilde{g}(y)=\begin{cases}g(y) & y\in W_{(0,1)}\\
0 & y\in W\setminus (W_{(-1,0)}\cup W_{(0,1)})\\
g(\mathcal{R}'(y)) & y\in W_{(-1,0)}
\end{cases}\]
is an element of $W^{1,1}(W)$ and for a.a $(y',y_d)=y\in W$ we have
\[\partial_i\widetilde{g}(y)=\begin{cases}
\partial_i g(y) & y\in W_{(0,1)}, \ i=1,\dots, d \\
(\partial_i g)(\mathcal{R}'(y))+ 2(\partial_d g)(\mathcal{R}'(y))\partial_i \mathfrak{f}(y') & (y',y_d)\in W_{(-1,0)} \ i=1,\dots, d-1 \\
-(\partial_d g)(\mathcal{R}'(y)) & (y',y_d)\in W_{(-1,0)}, \ i=d.
\end{cases}
\]
It is clear that $\widetilde{g}_{|W_{(0,1)}}$ is an element of $W^{1,1}(W_{(0,1)})$.
Next we show that $\widetilde{g}_{| W_{(-1,0)}}$ is an element of $W^{1,1}(W_{(-1,0)})$. Notice that $\widetilde{g}_{|W_{(-1,0)}}$ is given by the formula
\[\widetilde{g}_{|W_{(-1,0)}}(y)=g(\mathcal{R}'(y)),\]
where $y\in W_{(-1,0)}$. By assumption, the map $\mathcal{R'}$ is of class $C^1(W,W)$ and  $\mathcal{R'}(W_{(0,1)})= W_{(-1,0)}$. Moreover, $\mathcal{R'}$ is its own inverse, therefore $\mathcal{R}'_{| W_{(0,1)}}: W_{(0,1)}\to W_{(-1,0)}$ is an  invertable map  of class $C^1(W_{(0,1)},W_{(-1,0)})$.
Now by the  Theorem \ref{Integration by substitution formula}, we conclude that $\widetilde{g}_{|W_{(-1,0)}}$ is an element of $W^{1,1}(W_{(-1,0)})$ and
\[\partial_i\widetilde{g}_{|W_{(-1,0)}}(y)=\sum\limits_{j=1}^d(\partial_j g)(\mathcal{R}'(y))\partial_i\varrho'_j(y),\]
for $i=1,\dots, d$ and a.a $y\in W_{(-1,0)}$.
Since we know the explicit formulas for $\partial_i\varrho'_j$  (see (\ref{partial va}) and below) we rewrite the above formula as
\[\partial_i\widetilde{g}_{|W_{(-1,0)}}(y)=\begin{cases}
(\partial_i g)(\mathcal{R}'(y))+ 2(\partial_d g)(\mathcal{R}'(y))\partial_i \mathfrak{f}(y')  &  (y',y_d)\in W_{(-1,0)}, i=1,..., d-1 \\
-(\partial_d g)(\mathcal{R}'(y)) &  (y',y_d)\in W_{(-1,0)}, \ i=d.
\end{cases}
\]
Now we show that $\widetilde{g}$ is indeed an element of $W^{1,1}(W)$. First we will show that $\widetilde{g}$ is weakly differentiable on $W$. By Lemma \ref{characterization of weak differentiability} it suffices to show that for each $u\in C^1_C(W)$ and $i=1,\dots, d$ we have
\[\int\limits_W\widetilde{g}(x)\partial_iu(x)dx=-\int\limits_W\partial_i\widetilde{g}(x)u(x)dx.\]
In view of (\ref{definicja u1}), (\ref{definicja u2}), (\ref{definicja u3}) and (\ref{nosniki u_1}), (\ref{nosniki u_2 i u_3}), for any $i=1,\dots, d$ we have
 \begin{align}\label{equality fo intergrals of u}
     \int\limits_W\widetilde{g}(y)\partial_i u(y)dy &=\int\limits_{W_{(\varepsilon/2,1)}}\widetilde{g}(y)\partial_i u_{1,\varepsilon}(y)dy+\int\limits_{W_{(-3\varepsilon,3\varepsilon)}}\widetilde{g}(y)\partial_i u_{2,\varepsilon}(y)dy\\&+\int\limits_{W_{(-1,-\varepsilon/2)}}\widetilde{g}(y)\partial_i u_{3,\varepsilon}(y)dy\notag= I_1(\varepsilon)+I_2(\varepsilon)+ I_3(\varepsilon).
 \end{align}
We deal with each of $I_j(\varepsilon)$, for $j=1,2,3$ separately. Notice that, since $g\in W^{1,1}(W_{(0,1)})$ and $u_{1,\varepsilon}\in C_C^1(W_{(\varepsilon/2,1)})\subset C_C^1(W_{(0,1)})$,
\[I_1(\varepsilon)=\int\limits_{W_{(\varepsilon/2,1)}}\widetilde{g}(y)\partial_i u_{1,\varepsilon}(y)dy=\int\limits_{W_{(0,1)}}{g}(y)\partial_i u_{1,\varepsilon}(y)dy=-\int\limits_{W_{(0,1)}}\partial_i{g}(y) u_{1,\varepsilon}(y)dy.\]
Similarly, since $\widetilde{g}_{|W_{(-1,0)}}\in W^{1,1}(W_{(-1,0)})$ and $u_{-1,-\varepsilon}\in C^1_C(W_{(-1,-\varepsilon/2)})\subset C^1_C(W_{(0,1)})$ we have
\begin{align*}
    I_3(\varepsilon)&=\int\limits_{W_{(-1,-\varepsilon/2)}}\widetilde{g}(y)\partial_i u_{3,\varepsilon}(y)dy=\int\limits_{W_{(-1,0)}}\widetilde{g}_{|W_{(-1,0)}}(y)\partial_i u_{3,\varepsilon}(y)dy\\&=-\int\limits_{W_{(-1,0)}}\partial_i \widetilde{g}_{|W_{(-1,0)}}(y)u_{3,\varepsilon}(y).
\end{align*}
  The case of $I_2$ is slightly more involved.  First we rewrite it as a sum of two integrals
\begin{align*}
    I_2(\varepsilon) &=\int\limits_{W_{(-3\varepsilon,3\varepsilon)}}\widetilde{g}(y)\partial_i u_{2,\varepsilon}(y)dy=\int\limits_{W_{(-3\varepsilon,3\varepsilon)}}\widetilde{g}(y)\partial_i u(y) F_{-\varepsilon,\varepsilon}(y)dy\\&+\int\limits_{W_{(-3\varepsilon,3\varepsilon)}}\widetilde{g}(y) u(y)\partial_i F_{-\varepsilon,\varepsilon}(y)dy =
I_{2,1}(\varepsilon)+I_{2,2}(\varepsilon).
\end{align*}
If we denote by $C_u=\sup\limits_{y\in W}|\partial_i u(y)|$, then we get that for every $0<\varepsilon<\frac{1}{3}$ and $y\in W$,
\[
 |\widetilde{g}(y)\partial_i u(y) F_{-\varepsilon,\varepsilon}(y)|\leq C_u|\widetilde{g}(y)|.
 \]
In view of $\widetilde{g}_{|W_{(0,1)}}=g\in W^{1,1}(W_{(0,1)})$ and $\widetilde{g}_{|W_{(-1,0)}}=(g\circ \mathcal{R'})_{W_{(-1,0)}}\in W^{1,1}(W_{(-1,0)})$ we have $\widetilde{g}_{|W_{(-1,0)}}\in L^1(W_{(-1,0)})$ and $\widetilde{g}_{|W_{(0,1)}}\in L^1(W_{(0,1)})$. By the fact that $|W\setminus (W_{(-1,0)}\cup W_{(0,1)})|=0$ we get that $\widetilde{g}\in L^1(W)$ and since $W_{(-3\varepsilon,3\varepsilon)}\subset W$ we conclude $\widetilde{g}\in L^1(W_{(-3\varepsilon,3\varepsilon)})$.
Hence, by the Lebesgue Dominated Convergence Theorem, we have that
\[\lim\limits_{\varepsilon\to 0^+}I_{2,1}(\varepsilon)=0.\]
Now we want to deal with the term $I_{2,2}(\varepsilon)$. By (\ref{Jacobian of R'}), (\ref{rownosc w odbiciu dla d}), (\ref{rownosc w odbiciu dla i}), the fact that $\mathcal{{R'}}\circ\mathcal{R'}=Id_W$ and the integration by substitution formula, we have 
\begin{align}\label{nierownosc na I22}
  I_{2,2}(\varepsilon)&\notag=\int\limits_{W_{(-3\varepsilon,3\varepsilon)}}\widetilde{g}(y) u(y)\partial_i F_{-\varepsilon,\varepsilon}(y)dy
 \notag\\&=-\int\limits_{W_{(-3\varepsilon,3\varepsilon)}}\widetilde{g}(y) u(y)(\partial_iF_{-1,-\varepsilon}(y)+\partial_i F_{\varepsilon,1}(y))dy \notag\\&
 =-\int\limits_{W_{(-3\varepsilon,0)}}\widetilde{g}(y) u(y)\partial_i F_{-1,-\varepsilon}(y)dy-\int\limits_{W_{(0,3\varepsilon)}}\widetilde{g}(y)u(y)\partial_i F_{\varepsilon,1}(y)dy \notag\\&
 =-\int\limits_{W_{(-3\varepsilon,0)}}g(\mathcal{R}'(y))u(y)\partial_i F_{-1,-\varepsilon}(y)dy-\int\limits_{W_{(0,3\varepsilon)}}g(y)u(y)\partial_i F_{\varepsilon,1}(y)dy\notag\\&
 \overset{ \mathcal{R'}\circ\mathcal{R'}=Id_W}=-\int_{W_{(0,3\varepsilon)}}g(y)u(\mathcal{R'}(y))\partial_i F_{-1,-\varepsilon}(\mathcal{R'}(y))|\det J_{\mathcal{R'}}(y)|dy\notag\\&-\int\limits_{W_{(0,3\varepsilon)}}g(y)u(y)\partial_i F_{\varepsilon,1}(y)dy
  \notag\\&\overset{(\ref{Jacobian of R'}), \ (\ref{rownosc w odbiciu dla d}), \  (\ref{rownosc w odbiciu dla i})}=\int_{W_{(0,3\varepsilon)}}g(y)u(\mathcal{R'}(y))\partial_i F_{\varepsilon,1}(y)dy\notag\\&-\int\limits_{W_{(0,3\varepsilon)}}g(y)u(y)\partial_i F_{\varepsilon,1}(y)dy
= \int\limits_{W_{(0,3\varepsilon)}}g(y)\partial_i F_{\varepsilon,1}(y)(u(\mathcal{R'}(y))-u(y))dy.
\end{align}
Recall that for any $y\in W_{(0,3\varepsilon)}$,  $y=(y',y_d)$ and is of the form
\[(y',y_d)=(y', \mathfrak{f}(y')+tA)\]
for some $t\in(0,3\varepsilon)$, and also we have
\[\mathcal{R'}(y)=(y',\mathfrak{f}(y')-tA).\]
By the Mean Value Theorem, applied to the $d$-th variable on the interval $(\mathfrak{f}(y')-tA,\mathfrak{f}(y')+tA)$ for some $\theta\in (-1,1)$ we have
\[u(y)-u(\mathcal{R}'(y))=u(y',\mathfrak{f}(y')+tA)-u(y',\mathfrak{f}(y')-tA)=2A\partial_d u(y', \mathfrak{f}(y')+\theta tA)t.\]
Denote now $C'_u=\sup\limits_{y\in W}|\partial_d u(y)|$. Then for any $y\in W_{(0,3\varepsilon)}$ we have
\[|u(y)-u(\mathcal{R}'(y))|\leq 6 A C'_u \varepsilon.\]
And so by (\ref{i partial of F}), for $y\in W_{(0,3\varepsilon)}$ and $i=1,\dots, d-1$ we have
\[|\partial_i F_{\varepsilon,1}(y)(u(\mathcal{R'}(y))-u(y))|\leq 6A C'_u\varepsilon\sup\limits_{x\in Q_{d-1}(0,r')} |\partial_i \mathfrak{f}(x)|\frac{15}{A\varepsilon}=90C'_u\sup\limits_{x\in Q_{d-1}(0,r')} |\partial_i \mathfrak{f}(x)|,\]
and also by (\ref{d partial of F}) we have
\[|\partial_d F_{\varepsilon,1}(y)(u(\mathcal{R'}(y))-u(y))|\leq 6A C_u'\varepsilon\frac{15}{A\varepsilon}\leq 90C'_u.\]
And so, by (\ref{nierownosc na I22}) and the Lebesgue Dominated Convergence Theorem we conclude that
\[\lim\limits_{\varepsilon\to 0^+} I_{2,2}(\varepsilon)=0.\]
Therefore, 
\begin{align}\label{calka i22}
   \lim\limits_{\varepsilon\to 0^+}I_2(\varepsilon)=\lim\limits_{\varepsilon\to 0^+}I_{2,1}(\varepsilon)+\lim\limits_{\varepsilon\to 0^+}I_{2,2}(\varepsilon)=0 
\end{align}
Recall that  $C_u=\sup\limits_{y\in W}|\partial_i u(y)|<\infty$ and so $|\widetilde{g}(y)\partial_i u(y)|\leq C_u\widetilde{g}(y)$ for all $y\in W$. Hence, by the fact that $\widetilde{g}\in L^1(W)$, (\ref{equality fo intergrals of u}), (\ref{calka i22}), and the Lebesgue Dominated Convergence Theorem,
\begin{align*}
     \int\limits_W\widetilde{g}(y)\partial_i u(y)dy&=\lim\limits_{\varepsilon\to 0^+} (I_{1}(\varepsilon)+I_{2}(\varepsilon)+I_{3}(\varepsilon))=\lim\limits_{\varepsilon\to 0^+} I_{1}(\varepsilon)+\lim\limits_{\varepsilon\to 0^+}I_{2}(\varepsilon)+\lim\limits_{\varepsilon\to 0^+}I_{3}(\varepsilon)\\&=
    -\lim\limits_{\varepsilon\to 0^+}\int\limits_{W_{(0,1)}}\partial_i{g}(y) u_{1,\varepsilon}(y)dy-\lim\limits_{\varepsilon\to 0^+}\int\limits_{W_{(-1,0)}}\partial_i\widetilde{g}_{|W_{(-1,0)}}(y) u_{3,\varepsilon}(y)dy\\&=-\int\limits_W\left(\partial_i{g}(y)u(y)\chi_{W_{(0,1)}}(y)+ \partial_i\widetilde{g}_{|W_{(-1,0)}}(y)u(y)\chi_{W_{(-1,0)}}(y)\right)dy\\&=-\int\limits_W \partial_i\widetilde{g}(y)u(y)dy,
\end{align*}
and so $\widetilde{g}$ is weakly differentiable on $W$. Since both $\widetilde{g}$ and its weak derivatives are integrable on $W$, we conclude that $\widetilde{g}\in W^{1,1}(W)$.

 \end{proof}
 For any path $\gamma:[0,1]\to \R^d$ we denote its length by $l(\gamma)$. 
\begin{lemma}\label{lemat o krzywej}
    Let  $(W,  \{W_{(t_0,t_1)}\}_{(t_0,t_1)\subset (-1,1)}, \mathcal{R'})$  be the triple from Theorem \ref{reflection property}. 
     For any $x,y\in W$ we have
     \[|\mathcal{R'}(x)-\mathcal{R'}(y)|\leq C_1|x-y|,\]
     for some constant $C_1>0$ independent of points $x,y$. Moreover,
     there exist a family of paths $\gamma_{x,y}:[0,1]\to W$ indexed by $x,y\in W$, and a constant $C_2>0$ independent of $x,y$,  such that
     \begin{enumerate}
        \item[$(1)$]  $\gamma_{x,y}(0)=x$, $\gamma_{x,y}(1)=y$,
        \item[$(2)$] $\mathcal{R}'(\gamma_{x,y}(s))=\gamma_{\mathcal{R'}(x),\mathcal{R'}(y)}(s)$, for $s\in [0,1]$,
        \item[$(3)$]  if we write $x=(x',\mathfrak{f}(x')+t_x)$, $y=(y',\mathfrak{f}(y')+t_y)$ for some $t_x, t_y\in (-1,1)$ such that $t_x\leq t_y$, then for any $s\in [0,1]$ and any interval $(a,b)$ such that $[t_x,t_y]\subset(a,b)\subset(-1,1)$ we have
        \[\gamma_{x,y}(s)\in W_{(a,b)}\]
        and
        \[\mathcal{R'}(\gamma_{x,y}(s))\in W_{(-b,-a)},\]
        \item[$(4)$]  $l(\gamma_{x,y})\leq C_2|x-y|$,
        \item[$(5)$]  $l(\mathcal{R'}(\gamma_{x,y}))\leq C_1C_2|x-y|.$
  
     \end{enumerate}

\end{lemma}
\begin{proof}
  Take any $x=(x',x_d)\in W$ and $y=(y',y_d)\in W$, by the Mean Value Theorem we have 
  \[|\mathfrak{f}(x')-\mathfrak{f}(y')|=|\nabla \mathfrak{f}(x'+t_0(y'-x'))\cdot(x'-y')|,\]
  for some $t_0\in [0,1]$. Let
  
  \begin{align}\label{ograniszenie gradientu}
     C_f=\sup\limits_{x\in Q_{d-1}(0,r')}|\nabla \mathfrak{f}(x) |, 
  \end{align}
by assumption on $\mathfrak{f}$,  $C_f<\infty$.
 Hence 
  \begin{align}\label{stala dla krzywej}
      |\mathfrak{f}(x')-\mathfrak{f}(y')|\leq C_f|x'-y'|
  \end{align}
  for any $x',y'\in Q_{d-1}(0,r')$ and so for any $x,y\in W$,
  \begin{align*}
&|\mathcal{R'}(x)-\mathcal{R'}(y)|=\sqrt{|x'-y'|^2+|2(\mathfrak{f}(x')-\mathfrak{f}(y'))+(y_d-x_d)|^2}\\&\leq\sqrt{8|x'-y'|^2+ 8|\mathfrak{f}(x')-\mathfrak{f}(y')|^2+8|x_d-y_d|^2}\\&\leq 
\sqrt{8|x'-y'|^2+8C_f^2|x'-y'|^2+8|x_d-y_d|^2}\\&\leq
\sqrt{8|x'-y'|^2+8C_f^2|x'-y'|^2+8C_f^2|x_d-y_d|^2+8|x_d-y_d|^2}\\&=\sqrt{8|x-y|^2+8C_f^2|x-y|^2}=\sqrt{8+8C_f^2}|x-y|.
 \end{align*}
We rewrite it as
\begin{align}\label{23}
    |\mathcal{R'}(x)-\mathcal{R'}(y)|\leq C_1 |x-y|, 
\end{align}
where $C_1=\sqrt{8+8C_f^2}$, and $x,y\in W$.

  Letting $x,y\in W$,  $x=(x',x_d)=(x',\mathfrak{f}(x')+At_x)$ and $y=(y',y_d)=(y',\mathfrak{f}(y')+At_y)$ for some $t_x,t_y\in(-1,1)$. Define now a path $\gamma_{x,y} : [0,1]\to W$ by the formula
\begin{align}\label{formula na krzywa}
    \gamma_{x,y}(s)=((1-s)x'+sy', \mathfrak{f}((1-s)x'+sy')+(1-s)At_x+sAt_y),
\end{align}
for $s\in [0,1]$. It is clear that $\gamma_{x,y}(0)=(x',\mathfrak{f}(x')+At_x)=x$ and $\gamma_{x,y}(1)=y$ which is (1).  We also have for $s\in [0,1]$,

\begin{align}\label{odbicie i gamma}
 \mathcal{R'}(\gamma_{x,y}(s))&\notag= ((1-s)x'+sy', \mathfrak{f}((1-s)x'+sy')-(1-s)At_x-sAt_y)\\&\notag=((1-s)x'+sy', \mathfrak{f}((1-s)x'+sy')+(1-s)(-At_x)+s(-At_y))\\&=
 \gamma_{\mathcal{R'}(x),\mathcal{R'}(y)}(s),
\end{align}
which is (2).

  Condition (3) is quite obvious. In fact for each $s\in[0,1]$ we have $\gamma_{x,y}(s)\in S+A((t_y-t_x)s+t_x)e_d$, so for any interval $(a,b)\subset (-1,1)$ such that $[t_x,t_y]\subset (a,b)$, provided that $t_x\leq t_y$, we have
\[
    \gamma_{x,y}(s)\in W_{(a,b)} \ \ \ \text{and} \ \ \ 
    \mathcal{R'(}\gamma_{x,y}(s))\in W_{(-b,-a)}.
\]

  We will show now condition (4) that is the length of $\gamma_{x,y}$ is bounded by $C_2|x-y|$, where the constant $C_2$ is independent of the choice of $x,y$.  By the formula for the length of a parametric curve and the Mean Value Theorem, there exists $s_1\in (0,1)$ such that
\begin{align*}
   l(\gamma_{x,y})&=\int\limits_0^1\sqrt{|x'-y'|^2+|(\nabla f)(x'+(y'-x')s)\cdot(y'-x')+A(t_y-t_x)|^2}ds\\&=
  \sqrt{|x'-y'|^2+|(\nabla f)(x'+(y'-x')s_1)\cdot(y'-x')+A(t_y-t_x)|^2}.
\end{align*}

  Hence by the inequality (\ref{stala dla krzywej}),  Cauchy Schwartz Inequality and convexity of the function $u(t)=t^2$ we have
\begin{align*}
    l(\gamma_{x,y})&=\int\limits_0^1\sqrt{|x'-y'|^2+|(\nabla f)(x'+(y'-x')s)\cdot(y'-x')+A(t_y-t_x)|^2}ds\\
    &\overset{MVT}=
  \sqrt{|x'-y'|^2+|(\nabla f)(x'+(y'-x')s_1)\cdot(y'-x')+A(t_y-t_x)|^2}\\&\overset{CSI}\leq 
  \sqrt{|x'-y'|^2+2|(\nabla f)(x'+(y'-x')s_1)|^2|y'-x'|^2+2|A(t_y-t_x)|^2}\\&\overset{(\ref{ograniszenie gradientu})}\leq
  \sqrt{(2C_f^2+1)|y'-x'|^2+2|A(t_y-t_x)+\mathfrak{f}(y')-\mathfrak{f}(x')-(\mathfrak{f}(y')-\mathfrak{f}(x'))|^2}\\&\leq
  \sqrt{(2C_f^2+2)}\sqrt{\begin{aligned}&|y'-x'|^2+|\mathfrak{f}(y')-\mathfrak{f}(x')+A(t_y-t_x)|^2\\&+2|\mathfrak{f}(y')-\mathfrak{f}(x')+A(t_y-t_x)||\mathfrak{f}(y')-\mathfrak{f}(x')|+|\mathfrak{f}(y')-\mathfrak{f}(x')|^2\end{aligned} }\\&\overset{(\ref{odbicie i gamma})}\leq 
  \sqrt{(2C_f^2+2)}\sqrt{2|x-y|^2+2C_f|x-y||x'-y'|+C_f^2|x'-y'|^2}\\&\leq
  \sqrt{(2C_f^2+2)(C_f^2+2C_f+2)}|x-y|=C_2|x-y|.
\end{align*}
Hence we get condition (4) that is for $x,y\in W$,
\begin{align*}
   l(\gamma_{x,y})\leq C_2|x-y|,
\end{align*}
where $C_2= \sqrt{(2C_f^2+2)(C_f^2+2C_f+2)}$.
Now by (\ref{23}), we conclude that
\begin{align*}
 l(\mathcal{R'}(\gamma_{x,y}))=l(\gamma_{\mathcal{R'}(x),\mathcal{R'}(y)})\leq C_2 |\mathcal{R'}(x)-\mathcal{R'}(y)|\leq C_2C_1|x-y|,\end{align*}
 which shows (5) and finishes the proof.

\end{proof}
\begin{lemma}\label{extension of Phi}
  Let $(W,  \{W_{(t_0,t_1)}\}_{(t_0,t_1)\subset (-1,1)}, \mathcal{R'})$ be the triple from Theorem \ref{reflection property}.
For a MO function $\Phi$ on $W_{(0,1)}$,  define the function $\Phi^W:W\times [0,\infty)\to [0,\infty)$ by the formula
     \[\Phi^W(x,t)=
   \begin{cases} 
      \Phi(x,t) & x\in W_{(0,1)} \\
      0 & x\in S \\
      \Phi(\mathcal{R'}(x),t) & x\in W_{(-1,0)}. 
      
   \end{cases}
\]
Then $\Phi^W$ is a MO function. Moreover
\begin{enumerate}
    \item[$(1)$]  if $\Phi$ satisfies $\Delta_2$ condition, so does $\Phi^W$.
    \item[$(2)$]  Let $N=\left\lceil 2 C_1C_2\right\rceil$ and $N_1=\left\lceil 2C_1\right\rceil$, where constants $C_1, \ C_2$ are from the Lemma (\ref{lemat o krzywej}). If $\Phi$ satisfies condition {\rm(A1)} with constant $\beta$, then $\Phi^W$ satisfies {\rm(A1)} with the constant $\beta^{N+N_1}$.
\end{enumerate}
\end{lemma}

\begin{proof}
  Let $S, \ W, \ \Phi $ and $\mathcal{R'}$ be as in the assumptions. Since $\mathcal{R'}$ is of class $C^1(W,W)$, the function $\Phi(\cdot,t)\circ\mathcal{R'}$ is measurable on $W_{(-1,0)}$, for each $t\geq 0$. Moreover, since 
  \[W\setminus (W_{(-1,0)}\cup W_{(0,1)})=S\]
  and $|S|=0$ we conclude that $\Phi^W(\cdot,t)$ is measurable for each $t\geq 0$. On the other hand, if $x\in W_{(-1,0)}$, then $\mathcal{R'}(x)\in W_{(0,1)}$ and so $\Phi(\mathcal{R'}(x),\cdot)$  is a non-negative, convex function such that $\Phi(\mathcal{R'}(x),t)=0$ if and only if $t=0$. Hence $\Phi^W$ is a MO function on $W$.
  
   (1) Assume now that $\Phi\in\Delta_2$, then there exist $C>0$ and a non-negative $h\in L^1(W_{(0,1)}) $ such that for any $t\geq 0$ and a.e. $x\in W_{(0,1)}$,
\[\Phi(x,2t)\leq C\Phi(x,t)+h(x).\]
Define now $h_W:W\to[0,\infty)$ by the formula 
     \[h_W(x)=
   \begin{cases} 
      h(x) & x\in W_{(0,1)} \\
      0 &  x\in S \\
      h(\mathcal{R'}(x)) & x\in W_{(-1,0)}. 
    \end{cases}
\]
Since $\mathcal{R'}$ is a differentiable involution, we have that $|\det J_{\mathcal{R'}}(x)|=1$ for all $x\in W$ and so
\begin{align*}
   \int\limits_W h_W(x)dx&=\int\limits_{W_{(1,0)}} h_W(x)dx+\int\limits_{W_{(1,0)}}  h_W(\mathcal{R'}(x))dx\\&\leq \|h\|_{1}+\int\limits_{W_{(1,0)}} h_W(x)dx|\det J_{\mathcal{R'}}(x)|dx\leq 2 \|h\|_1. 
\end{align*}

Hence $h_W$ is a non-negative element of $L^1(W)$. Taking now any $x\in W_{(0,1)}$  and $t\geq 0$, 
\[\Phi^W(x,2t)=\Phi\left(x,2t\right)\leq C\Phi(x,t)+h(x)=C\Phi^W(x,t)+h_W(x).\]
Similarly, if $x\in W_{(-1,0)}$, then 
\[\Phi^W(x,2t)=\Phi\left(\mathcal{R'}(x),2t\right)\leq C\Phi\left(\mathcal{R'}(x),t\right)+h(\mathcal{R'}(x))=C\Phi^W(x,t)+h_W(x).\]
Therefore $\Phi^W$ satisfies $\Delta_2$ condition.

  (2) Now assume that $\Phi$ satisfies the condition {\rm (A1)} with the constant $\beta$ and let $N=\lceil2 C_1C_2\rceil$, where constants $C_1, \ C_2$ are from Lemma \ref{lemat o krzywej}. We will show that $\Phi^W$ satisfies  condition (A1). 

  Recall that, if $\Phi$ satisfies condition (A1), in our case on set $W_{(0,1)}$, then there exist $\beta, \delta\in(0,1)$ such that for all open balls $B$ with $|B|<\delta$ and almost all $x,y\in B\cap W_{(0,1)}$ we have
\[\beta\Phi^{-1}(x, t)\leq \Phi^{-1}(y,t),\]
where $t\in \left[1,\frac{1}{|B|} \right]$.

  Let $S_\Phi\subset W$ be the set of measure $0$ such that the above inequality holds for all $x,y\in \left(B\cap W_{(0,1)}\right)\setminus S_\Phi$, provided that $B$ is an open ball with $|B|<\delta$. Note that, since for all $x\in W$ we have $|\det J_{\mathcal{R'}}(x)|=1$, it follows that the set $\mathcal{R'}(S_\Phi)$ is also of measure $0$. We will show now that for any open ball $B$ with $|B|<\delta$ and any $x,y\in (W_{(-1,0)}\cup W_{(0,1)})\cap B\setminus (S_\Phi \cup \mathcal{R'}(S_\Phi)) $ we have
\[\beta^{N+2} \left(\Phi^W\right)^{-1}(x,t)\leq \left(\Phi^W\right)^{-1}(y,t)\]
for $t\in\left[1,\frac{1}{|B|}\right]$. In view of $|S\cup S_\Phi\cup\mathcal{R'}(S_\Phi)|=0$ we will conclude that $\Phi^W$ satisfies (A1) condition.

  Take any open ball $B$ with $|B|<\delta$ and let $x,y\in (W_{(-1,0)}\cup W_{(0,1)})\cap B\setminus (S_\Phi \cup \mathcal{R'}(S_\Phi)) $. We have three cases
\begin{enumerate}
\item[$1^0$] $x,y\in W_{(0,1)}$,

\item[$2^0$] $x,y\in W_{(-1,0)}$,

\item[$3^0$] $x\in W_{(0,1)}$ and $y\in W_{(-1,0)}$.
\end{enumerate}

  In the first case, for any $t\geq 0$ we have that $\left(\Phi^W\right)^{-1}(x,t)=\Phi^{-1}(x,t)$ and $\left(\Phi^W\right)^{-1}(y,t)=\Phi^{-1}(y,t)$ and so, since $\Phi$ satisfies (A1) and $N, N_1\in \N$ we have
\begin{align*}
   \beta^{N+N_1} \left(\Phi^W\right)^{-1}(x,t)&=\beta^{N+N_1}\Phi^{-1}(x,t)\leq \beta^{N+N_1-1} \Phi^{-1}(y,t)<\Phi^{-1}(y,t)\\&= \left(\Phi^W\right)^{-1}(y,t). 
\end{align*}

  For the second case, by Lemma \ref{lemat o krzywej} there exists a path $\gamma_{x,y}:[0,1]\to W$ such that $\gamma_{x,y}(0)=x$ and $\gamma_{x,y}(1)=y$ and $\gamma_{x,y}(s)\in W_{(-1,0)}$ for each $s\in [0,1]$. By the same lemma we have that \[\gamma_{\mathcal{R'}(x),\mathcal{R'}(y)}(s)=\mathcal{R'}(\gamma_{x,y}(s))\in W_{(0,1)}\]
for each $s\in [0,1]$. Therefore the points $\mathcal{R'}(x)$ and $\mathcal{R'}(y)$ are connected by the path $\gamma_{\mathcal{R'}(x),\mathcal{R'}(y)}$ that is totally contained in $W_{(0,1)}$. Moreover by Lemma \ref{lemat o krzywej}, we have that for any $x,y\in W$,
\[l(\gamma_{\mathcal{R'}(x),\mathcal{R'}(y)})\leq C_1C_2|x-y|.\]
Since both $x,y\in B$ we have that 
\[|x-y|\leq 2\sigma_d^{-1/d}|B|^{1/d},\]
and so 
\[l(\gamma_{\mathcal{R'}(x),\mathcal{R'}(y)})\leq 2C_1C_2\sigma_d^{-1/d}|B|^{1/d},\]
where $\sigma_d$ is defined by (\ref{1}). 
Hence the path $\gamma_{\mathcal{R'}(x),\mathcal{R'}(y)}$ can be covered with a chain  of $N=\left\lceil 2C_1C_2\right\rceil$ open balls $B_1,\dots,B_N$, each of measure $|B|$, such that each two successive balls $B_i$, $B_{i+1}$ have non-empty intersection and $\mathcal{R'}(x)\in B_1$, $\mathcal{R'}(y)\in B_N$.

  Set $x_1=\mathcal{R'}(x)$ and $x_{N+1}=\mathcal{R'}(y)$, and for each $i=2,\dots, N$ choose a point $x_i\in B_{i-1}\cap B_{i}\cap W_{(0,1)}\setminus S_\Phi $. We observe that for each $i=1,\dots, N$ we have $x_{i}, x_{i+1}\in B_{i} $. Now notice that $\left(\Phi^W\right)^{-1}(x, t)=\Phi^{-1}(\mathcal{R'}(x),t)$ and   $\left(\Phi^W\right)^{-1}(y, t)=\Phi^{-1}(\mathcal{R'}(y),t)$. Finally, take any $t\in \left[1,\frac{1}{|B|}\right]$ and since for any $i=1,\dots, N$ we have $x_i,x_{i+1}\in B_{i}$ and $|B_{i}|=|B|$, by successive $N$ times application of condition (A1) we arrive at
\begin{align*}
    \beta^{N+N_1}(\Phi^W)^{-1}(x,t)&=\beta^{
    N+N_1}\Phi^{-1}(\mathcal{R'}(x),t)=\beta^{
    N+N_1}\Phi^{-1}(x_1,t)\leq \beta^{N+N_1-1}\Phi^{-1}(x_2,t)\\&\leq \beta^{N+N_1-2}\Phi^{-1}(x_3,t)\leq\dots\\&\leq \beta^{N+N_1-i}\Phi^{-1}(x_{i+1},t)\leq \beta^{N+N_1-i-1}\Phi^{-1}(x_{i+2},t)\leq\dots\\&\leq
    \beta^{N_1}\Phi^{-1}(x_{N+1},t)=\beta^{N_1}\Phi^{-1}(\mathcal{R'}(y),t)<\left(\Phi^W\right)^{-1}(y,t).
\end{align*}

  For the third case let us take the path $\gamma_{x,y}$ from Lemma \ref{lemat o krzywej} (see (\ref{formula na krzywa})). Since $x\in  B\cap W_{(0,1)}$, and $y\in B\cap W_{(-1,0)}$ there exist $t_x\in (0,1)$ and $t_y\in(-1,0)$ such that $x=(x',\mathfrak{f}(x')+At_x)$ and $y=(y',\mathfrak{f}(y')+At_y)$. Hence, by continuity of $\gamma_{x,y}$ there exists $s_0\in [0,1]$ such that $((1-s_0)x'+s_0y',\mathfrak{f}((1-s_0)x'+s_0y'))=\gamma_{x,y}(s_0)\in S$, let us denote $z=\gamma_{x,y}(s_0)$. By Lemma \ref{lemat o krzywej} we have for any $x,y\in W$,
\[l(\gamma_{x,y})\leq C_1|x-y|,\]
and since $x,y\in B$ we conclude that
\[l(\gamma_{x,y})\leq C_1|x-y|\leq 2C_1\sigma_d^{-1/d}|B|^{1/d}.\]
Hence, the path $\gamma_{x,y}$ can be covered by a chain of $N_1=\lceil 2C_1\rceil$ open balls, $B_1,\dots,B_{N_1}$ such that each two consecutive balls $B_i, B_{i+1}$ have a non-empty intersection. Let $K$ be the smallest integer such  that $z\in B_{K}$, it follows that we have that $K\leq N_1$.  
 Consider now the path $\gamma_{z,y}$.  Since $z\in S$ we have that $z=(z',\mathfrak{f}(z'))$, and for any $s\in [0,1]$, 

\[\gamma_{z,y}(s)=((1-s)z'+sy',\mathfrak{f}((1-s)z'+sy')+At_ys).\]
Since $y\in W_{(-1,0)}$ we conclude that $\gamma_{z,y}(s)\in W_{(-1,0)}$ for $s\in (0,1]$, but  $\gamma_{z,y}(0)=z\in S$ and $S\cap W_{(-1,0)}=\emptyset$. Consequently for $s\in [0,1]$, 
\[\gamma_{z,y}(s)\in S\cup W_{(-1,0)} \text{ and }\mathcal{R'}(\gamma_{z,y}(s))\in S\cup W_{(0,1)}.\]  On the other hand, since $z\in S$ and $z=\gamma_{x,y}(s_0)$ we have that 
\begin{align}\label{kombinacja wypukla x i y}
    (1-s_0)At_x+s_0At_y=0.
\end{align}
Hence,
\begin{align*}
  z&=\gamma_{x,y}(s_0)=((1-s_0)x'+s_0y',\mathfrak{f}((1-s_0)x'+s_0y')+(1-s_0)At_x+ s_0At_y)\\&=((1-s_0)x'+s_0y',\mathfrak{f}((1-s_0)x'+s_0y')),  \end{align*}
and so for $s\in [0,1]$ we have
\begin{align*}
    \gamma_{z,y}(s)&=((1-s)z'+sy',\mathfrak{f}((1-s)z'+sy')+sAt_y)\\&\overset{(\ref{kombinacja wypukla x i y})}=((1-s)((1-s_0)x'+s_0y')+sy',\\&\mathfrak{f}((1-s)((1-s_0)x'+s_0y')+sy')+(1-s)((1-s_0)At_x+s_0At_y)+sAt_y)
    \\&=((1-(s+s_0-ss_0))x'+(s+s_0-ss_0)y',
    \\&\mathfrak{f}((1-(s+s_0-ss_0))x'+(s+s_0-ss_0)y') \\&+(1-(s+s_0-ss_0))At_x+(s+s_0-ss_0)At_y)\\&
    =\gamma_{x,y}(s+s_0-ss_0)=\gamma_{x,y}((1-s)s_0+s).
\end{align*}
Since $s,s_0\in [0,1]$ we conclude that $((1-s)s_0+s)\in [s_0,1]$ and so 
\[\gamma_{z,y}([0,1])=\gamma_{x,y}([s_0,1]).\]
Clearly $\gamma_{z,y}$ is a part of the path $\gamma_{x,y}$. By Lemma \ref{lemat o krzywej} (5), we have that \[l(\mathcal{R'}(\gamma_{z,y}))\leq l(\mathcal{R'}(\gamma_{x,y}))\leq 2C_1C_2\sigma_d^{-1/d}|B|^{1/d}.\]
Similarly to the second case, we can find a chain of  $N=\lceil 4C_1C_2\rceil$  open balls $B'_{K+1},\dots,B'_{N+K}$, that covers the path $\mathcal{R'}(\gamma_{z,y})=\gamma_{z,\mathcal{R'}(y)}$, each of measure $|B|$ such that $z\in B'_{K+1}$, $\mathcal{R'}(y)\in B'_{N+K}$ and for $i={K+1},\dots,N+K-1$ each two successive balls $B'_i, B'_{i+1}$ have non-empty intersection. Since $z\in S\cap B_K\cap B'_{K+1}$ there exists a point $x_{K+1}\in W_{(0,1)}\cap B_K\cap B'_{K+1}\setminus S_\Phi$. Let now $x_1=x$ and for each $i=2,\dots, K$ choose a point $x_i\in B_{i-1}\cap B_{i}\cap W_{(0,1)}\setminus S_\Phi$, and for $i=K+2,\dots, N+K$ we choose a point $x_i\in B'_{i-1}\cap B'_{i}\cap W_{(0,1)}\setminus S_\Phi$ and finally we set  $x_{N+K+1}=\mathcal{R'}(y)$. We notice that $(\Phi^W)^{-1}(x,t)=\Phi^{-1}(x,t)$ and $(\Phi^W)^{-1}(y,t)=\Phi^{-1}(\mathcal{R'}(y),t)$. Now we take any $t\in \left[1,\frac{1}{|B|}\right]$ and argue as in the second case,
\begin{align*}
    \beta^{N+N_1}(\Phi^W)^{-1}(x,t)&=\beta^{
    N+N_1}\Phi^{-1}(x_1,t)\leq\beta^{
    N+K}\Phi^{-1}(x_1,t)\leq\beta^{
    N+K-1}\Phi^{-1}(x_2,t)\\&\leq \beta^{N+K-2}\Phi^{-1}(x_3,t)\leq\dots\leq \beta^{N+K-i}\Phi^{-1}(x_{i+1},t)\\&\leq \beta^{N+K-i-1}\Phi^{-1}(x_{i+2}.t)\leq\dots\overset{i=N+K-1}\leq
    \Phi^{-1}(x_{N+K+1},t)\\&=\Phi^{-1}(\mathcal{R'}(y),t)=\left(\Phi^W\right)^{-1}(y,t).
\end{align*}
Since all three cases are covered this ends the proof.
 \end{proof}
\begin{lemma}\label{Phi i rigid motion}
Let $V,W$ be open sets in $\R^d$ such that there exists a rigid motion $T$ on $\R^d$ with $T(V)=W$. Let $\Phi$ be a MO function on $V$ and define $\Phi_W:=\Phi\circ T^{-1}$. Then $\Phi_W$ is a MO function on $W$ and
\begin{enumerate}
    \item[$(1)$]  if $\Phi$ satisfies $ \Delta_2$ condition, then so does $\Phi_W$,
    \item[$(2)$]  if $\Phi$ satisfies {\rm (A1)}, then $\Phi_W$ also satisfies {\rm (A1)},
    \item[$(3)$]  the operator $G_T :W^{1,\Phi_W}(W)\to W^{1,\Phi}(V)$, defined by the formula
    \[(G_T u)(x)= u(T(x)),\]
for $u\in W^{1,\Phi_W}$ and $x\in W$, is an isomorphism.
\end{enumerate}
\end{lemma}
\begin{proof}
Let $V,W$ and $T$ be as in the assumption. Recall that, since $T$ is a rigid motion, we have 
\[T(x)= R(x)+c,\]
for every $x\in \R^d$, where $c\in\R^d$ and the map $R:\R^d\to\R^d$ is a rotation about the origin and is given by the formula
\[R(x)= R\cdot x=Rx,\]
for $x\in \R^d$. Here $R=(r_{ij})_{i,j=1}^d$ is a rotation matrix and the multiplication $\cdot$ is the usual matrix multiplication and $c\in \R^d$. Hence, for every $x\in\R^d$,
\[J_T(x)=R.\]
Since $R$ is a rotation matrix, by definition (see \cite[p. 39]{groove}),
\[\det J_T(x)=1.\]
Now taking any $x\in W$ and letting $y= T^{-1}(x)$, we have $y\in V$. For any $t\geq 0$, $x\in W$,
  \[\Phi_W(x,t)=\Phi(T^{-1}(x),t)=\Phi(y,t).\]
Since $\Phi$ is a MO function on $V$ and $T$ is continuous, we conclude that $\Phi_W$ is a MO function on $W$.

  Assume now that $\Phi\in\Delta_2$. Then there exist $C>0$ and a non-negative function $h\in L^1(V)$, such that for a.a $x\in V$ and $t\geq 0$ we have
\[\Phi(x,2t)\leq C\Phi(x,t)+h(x).\]
Let us define $h_W$ by the formula
\[h_W(x)= h(T^{-1}(x))\]
for  $x\in W$.
We have
\[\int\limits_W h(T(x))dx=\int\limits_V h(x)|\det J_T(x)|dx=\int\limits_V h(x)dx,\]
hence $h_W$ is an non-negative element of $L^1(W)$. For any $x\in W$ and  $y=T^{-1}(x)$, then we have $y\in V$ and
\begin{align*}
   \Phi_W(x,2t)&=\Phi(T^{-1}(x),2t)=\Phi(y,2t)\leq C\Phi (y, t)+h(y)\\&=
   C\Phi(T^{-1}(x),t)+h(T^{-1}(x))=C\Phi_W(x,t)+h_W(x). 
\end{align*}
We conclude that $\Phi_W$ satisfies $\Delta_2$.

  Assume now that $\Phi$ satisfies (A1) condition. Then  there exist constants $\beta,\delta \in (0,1)$ such that for all open balls $B$ with $|B|<\delta$ and almost every $x,y\in B\cap V$ we have for all $t\in \left[\Phi^{-1}(y,1) , \Phi^{-1}\left(y,\frac{1}{|B|}\right)\right]$,
\[\Phi(x,\beta t)\leq \Phi(y,t).\]
Take any ball $B$ with $|B|<\delta$. Since $T$ is a rigid motion $T^{-1}(B)$ is also a ball and $|T^{-1}(B)|=|B|$. Let us take any $x,y\in B\cap W$, then $T^{-1}(x),T^{-1}(y)\in T^{-1}(B)\cap V$ and for all $t\geq 0$, we have $\Phi^{-1}_W(y,t)=\Phi^{-1}(T^{-1}(y),t)$. For any $t\in \left[\Phi_W^{-1}(y,1) , \Phi_W^{-1}\left(y,\frac{1}{|B|}\right)\right]$, we have
\[\Phi_W(x,\beta t)=\Phi(T^{-1}(x),\beta t)\leq \Phi(T^{-1}(y),t)=\Phi_W(y,t)\]
and so $\Phi_W$ satisfies (A1) condition.

  Finally, let us show that $G_T$ is an isomorphism between $W^{1,\Phi_W}(W)$ and $W^{1,\Phi}(V)$. Take any $u\in W^{1,\Phi}(W)$ and $\lambda>0$ such that $I_{\Phi_W}(\lambda u)<\infty$. Then
\begin{align*}
    I_\Phi(\lambda G_Tu)&:=\int\limits_V \Phi(x,\lambda |u(T(x))|)dx=\int\limits_W\Phi(T^{-1}(x
    ),\lambda |u(x)|)|\det J_{T^{-1}}(x)|dx
    \\&=\int\limits_W\Phi_W(x,\lambda|u(x)|)dx=I_{\Phi_W}(\lambda u)<\infty.
\end{align*}
Hence $\|G_Tu\|_\Phi=\|u\|_{\Phi_W}$ and so $G_T$ is a linear isometry between $L^\Phi(V)$ and $L^{\Phi_W}(W)$.

  We will show now that $G_T u$ is a weakly differentiable function on $V$. First take any $v\in C_C^\infty(V)$. Since for all $x\in\R^d$ we have $J_{T^{-1}}(x)=R^{-1}$ and it is well known that an inverse of a rotation matrix $R$ is its transpose, $R^{-1}=R^T$,  we have 
\begin{align*}
    \partial_i (v\circ T^{-1})(x)=\sum\limits_{j=1}^d\partial_j v(T^{-1}(x))\partial_iT^{-1}_j(x)=\sum\limits_{j=1}^d\partial_j v(T^{-1}(x))r_{ij},
\end{align*}
$i=1,\dots, d$ and $x\in W$, where $(r_{ij})_{i,j}^d$ is a rotation matrix.
Hence we have $\nabla (v\circ T^{-1})(x)=R(\nabla v(T^{-1}(x)))$ and so $\nabla v(T^{-1}(x))=R^{T}(\nabla(v\circ T^{-1})(x))$,  in view of $R^{-1}=R^T$. Consequently, for each $i=1,\dots,d$ we have $\partial_iv(T^{-1}(x))=\sum\limits_{j=1}^d\partial_j(v\circ T^{-1})(x)r_{ji}$. Going back to $G_Tu$, $u$ is weakly differentiable, and so  for each $v\in C_C^\infty (V)$ and  for $i=1,\dots,d$
we have
\begin{align*}
  \int\limits_{V}    G_Tu(x)\partial_i v(x)dx &=\int\limits_{V}    u(T(x))\partial_i v(x)dx= \int\limits_W u(x)\partial_iv(T^{-1}(x))|\det J_{T}(x)|dx\\&=
  \int\limits_W u(x)\partial_iv(T^{-1}(x))dx=\int\limits_W u(x)\sum\limits_{j=1}^d \partial_j (v\circ T^{-1})(x)r_{ji}dx\\&=
  -\int\limits_W \left(\sum\limits_{j=1}^d \partial_j u(x)r_{ji}\right)(v\circ T^{-1})(x)dx\\&=-\int\limits_V \left(\sum\limits_{j=1}^d \partial_j u(T(x))r_{ji}\right)v(x)dx.
\end{align*}
Hence $G_Tu$ is weakly differentiable on $V$ and for each $i=1,\dots,d$ we have
\[\partial_i (G_Tu)(x)=\sum\limits_{j=1}^d \partial_j u(T(x))r_{ji}=\sum\limits_{j=1}^d(\partial_j u\circ T)(x)r_{ji},\]
for a.a $x\in V$. Notice that, since $R$ is a rotation matrix we have that $|r_{ji}|\leq 1$ and so for  $i=1,\dots,d$, a.a. $x\in V$,
\[|\partial_i (G_Tu)(x)|\leq\sum\limits_{j=1}^d|(\partial_ju\circ T)(x)|.\]
Now let $\lambda>0$ be such that for all $j=1,\dots, d$ we have
$I_{\Phi_W}(\lambda \partial_j u)<\infty$. For any $i=1,\dots, d$,
\begin{align*}
    I_{\Phi}\left(\frac{\lambda}{d} \partial_iG_Tu\right)&\leq I_{\Phi}\left(\frac{\lambda}{d}\sum\limits_{j=1}^d|(\partial_j u)\circ T|\right)\leq\frac{1}{d} \sum\limits_{j=1}^d I_\Phi(\lambda|(\partial_ju)\circ T|)\\&= 
    \frac{1}{d}\sum\limits_{j=1}^d\int\limits_{V} \Phi (x,\lambda|\partial_ju(T(x))|)dx=\frac{1}{d}\sum\limits_{j=1}^d\int\limits_{W} \Phi (T^{-1}(x),\lambda|\partial_j u(x)|)dx
    \\&=\frac{1}{d}\sum\limits_{j=1}^d I_{\Phi_W}(\lambda\partial_j u)<\infty.\end{align*}
    Taking $\lambda=1/\left(\sum\limits_{j=1}^d\|\partial_ju\|_{\Phi_{W}}\right)$ we arrive at
    \[I_{\Phi}\left(\frac{\lambda}{d} \partial_iG_Tu\right)\leq 1.\]
    Hence, we have that $\|\partial_iG_Tu\|_\Phi\leq d\sum\limits_{j=1}^d\|\partial_j u\|_{\Phi_W}$ for $i=1,\dots, d$ and so
    \[\|G_Tu\|_{W^{1,\Phi}(V)}\leq (d^2 +1) \|u\|_{W^{1,\Phi_W}(W)}.\]
We conclude that $G_T$ is continuous. Notice also that
\[G_T^{-1}=G_{T^{-1}}\]
and since $T^{-1}$ is a rigid motion, the same reasoning as above leads us to conclude that $G_T^{-1}$ is also continuous, and thus $G_T$ is an isomorphism. 
\end{proof}

\begin{theorem} (Local extension)\label{local extension}
  Let $\Omega$ be an open set in $\R^d$ such that its boundary is of class $C^1$. Let $\Phi$ be a MO function on $\Omega$ satisfying both  {\rm (A1)} and $\Delta_2$ property. For any $x_0$ let $(V, \{V_{(t_0,t_1)}\}_{(t_0,t_1)\subset (-1,1)}, \mathcal{R})$ and $(W, \{W_{(t_0,t_1)}\}_{(t_0,t_1)\subset (-1,1)}, \mathcal{R}')$ be the triples from Theorem \ref{reflection property}. Let also $T$ be the rigid motion from the same theorem.
  Define the function
  $\widetilde{\Phi}:W_{(0,1)}\times[0,\infty)\to [0,\infty)$ by the formula 
  \[\widetilde{\Phi}(y,t)=\Phi(T^{-1}(y),t),\]
  where $y\in W_{(0,1)}$ ant $t\geq 0$.
 We  extend $\widetilde{\Phi}$ to $\widetilde{\Phi}^W$ on $W$  as in Lemma \ref{extension of Phi} and define the function $\Phi_V:V\times [0,\infty)\to [0,\infty)$ via the formula
  \[\Phi_V(x,t)=\widetilde{\Phi}^W(T(x),t),\]
  for $t\geq 0$ and $x\in V$. Then, for any compactly supported $u\in W^{1,\Phi}(\Omega\cap V)$, there exists an extension $\widetilde{u}$ defined by the formula
   \[\widetilde{u}(x):=
   \begin{cases} 
      u(x) & x\in V_{(0,1)} \\
      0 &  x\in\bd(\Omega)\cap V \\
      u(\mathcal{R}(x)) & x\in V_{(-1,0)}, 
   \end{cases}
\]
 satisfying $\widetilde{u}\in W^{1,\Phi_V}(V)$ and $\widetilde{u}$ is compactly supported. 
\end{theorem}
\begin{proof}
Fix $x_0\in\bd (\Omega)$. Let $(V, \{V_{(t_0,t_1)}\}_{(t_0,t_1)\subset (-1,1)}, \mathcal{R})$,  $(W,  \{W_{(t_0,t_1)}\}_{(t_0,t_1)\subset (-1,1)}, \mathcal{R'})$ be the triples from Theorem \ref{reflection property} and $T$ the rigid motion from the same theorem. By Lemma \ref{Phi i rigid motion} the function $\widetilde{\Phi}$ satisfies both $\Delta_2$ and (A1) conditions. By Lemma \ref{extension of Phi} the function $\widetilde{\Phi}^W$ also satisfies $\Delta_2$ and (A1) conditions. Moreover for $y\in W$ and $t\geq 0$,
   \[\widetilde{\Phi}^W(y,t)=
   \begin{cases} 
      \Phi(T^{-1}(y),t) & y\in W_{(0,1)} \\
      0 &  y\in S \\
      \Phi(T^{-1}\mathcal({R'}(y)),t) & y\in W_{(-1,0)}. 
         \end{cases}\]
    By construction of $V$ and $W$, the rigid motion $T$ maps $V$ onto $W$. Therefore any $x\in V$ can be uniquely written as $x=T^{-1}(y)$ for some $y\in W$. Since $\mathcal{R}=T^{-1}\circ \mathcal{R'}\circ T$, we arrive at

     \[\Phi_V(x,t)=
   \begin{cases} 
      \Phi(x,t) & x\in V_{(0,1)} \\
      0 &  x\in\bd(\Omega)\cap V \\
      \Phi(\mathcal{R}(x),t) & x\in V_{(-1,0)},
      
   \end{cases}\]
where $x\in V$, and $t\geq 0$.

  Taking any compactly supported $u\in W^{1,\Phi}(\Omega\cap V)$, by the fact that $\Phi$ satisfies (A1) and by Proposition \ref{zanurzenie w L1} we have that $u\in L^1(\Omega\cap V)$ and for each $i=1,\dots, d$, $\partial_i u\in L^1(\Omega\cap V)$. Therefore $u\in W^{1,1}(\Omega)$. We define $g: W_{(0,1)}\to \C$ by the formula
\[g(y)=u(T^{-1}(y)), \text{ for } y\in W_{(0,1)}.\]
In view of $V\cap\Omega= V_{(0,1)}$ and the fact that $T$ is a rigid motion and $T(V_{(0,1)})=W_{(0,1)}$, by Theorem \ref{Integration by substitution formula} we conclude that $g\in W^{1,1}(W_{(0,1)})$. Hence by Lemma \ref{local extension for regular domains} the function
\[\widetilde{g}(y)=\begin{cases}g(y) & y\in W_{(0,1)}\\
0 & y\in W\setminus (W_{(-1,0)}\cup W_{(0,1)})\\
g(\mathcal{R}'(y)) & y\in W_{(-1,0)}
\end{cases}\]
is an element of $W^{1,1}(W)$ and for a.a $y\in W$ we have
\[\partial_i\widetilde{g}(y)=\begin{cases}
\partial_i g(y) & (y',y_d)\in W_{(0,1)}, \ i=1,\dots, d \\
(\partial_i g)(\mathcal{R}'(y))+ 2(\partial_d g)(\mathcal{R}'(y))\partial_i \mathfrak{f}(y') & (y',y_d)\in W_{(-1,0)} \ i=1,\dots, d-1 \\
-(\partial_d g)(\mathcal{R}'(y)) & (y',y_d)\in W_{(-1,0)}, \ i=d,
\end{cases}
\] 
where $\mathfrak{f}$ is the function associated with $x_0$ via the definition of $\bd(\Omega)$ being of class $C^1$ (Definition \ref{boundary of class C^m}). By (\ref{ograniczenie pochdnych f na kostce r'}) we have that $|\partial_i \mathfrak{f}(y')|<C$ for some constant $C>0$, all $i=1,\dots,d-1$, and any $y=(y',y_d)\in W$. Therefore, for any $\lambda>0$ and $i=1,\dots,d-1$, using a substitution for $\mathcal{R'}(y)$,
\begin{align*}
    I_{\widetilde{\Phi}^W}(\lambda\partial_i\widetilde{g})&\leq\int\limits_{W_{(0,1)}}\widetilde{\Phi}(y,\lambda |\partial_ig(y)|)dy+\frac{1}{2}\int\limits_{W_{(-1,0)}}\widetilde{\Phi}(\mathcal{R'}(y), 2\lambda|\partial_i g(\mathcal{R'}(y))|)dy \\&+\frac{1}{2}\int\limits_{W_{(-1,0)}}\widetilde{\Phi}(\mathcal{R'}(y), 4C\lambda|\partial_d g(\mathcal{R'}(y))|)dy
    =\int\limits_{W_{(0,1)}}\widetilde{\Phi}(y,\lambda |\partial_ig(y)|)dy\\&+\frac{1}{2}\int\limits_{W_{(0,1)}}\widetilde{\Phi}(y,2\lambda |\partial_ig(y)|)dy+\frac{1}{2}\int\limits_{W_{(0,1)}}\widetilde{\Phi}(y,4C\lambda |\partial_dg(y)|)dy
    \\&=\int\limits_{V_{(0,1)}}{\Phi}(x,\lambda |\partial_ig(T(x))|)dx\\&+\frac{1}{2}\int\limits_{V_{(0,1)}}{\Phi}(x,2\lambda |\partial_ig(T(x))|)dx+\frac{1}{2}\int\limits_{V_{(0,1)}}{\Phi}(x,4C\lambda |\partial_dg(T(x))|)dx.
\end{align*}
Let $T(x)=R(x)+c$, where $R=(r_{ij})_{i,j=1}^d$ is a rotation matrix and $c\in \R^d$. By definition of $g$ for $x\in V$, \[\partial_ig(T(x))=\sum\limits_{j=1}^d\partial_{j}u(x)r_{ji},\]
for all $i=1,\dots,d$. Since $R$ is a rotation matrix we have $|r_{ji}|\leq 1$ for any $1\leq i,j\leq d$. Hence, by $\Delta_2$ condition, for $i=1,\dots,d-1$,

\begin{align*}
     I_{\widetilde{\Phi}^W}(\lambda\partial_i\widetilde{g})&\leq \frac{1}{d}\sum\limits_{j=1}^d \int_{V_{(0,1)}}\Phi (x,d\lambda |\partial_j u(x)|)dx +\frac{1}{2d}\sum\limits_{j=1}^d \int_{V_{(0,1)}}\Phi (x,2d\lambda |\partial_j u(x)|)dx\\&+\frac{1}{2d}\sum\limits_{j=1}^d \int_{V_{(0,1)}}\Phi (x,4Cd\lambda |\partial_j u(x)|)dx<\infty.
\end{align*}
For $i=d$ similarly, we have
\begin{align*}
    I_{\widetilde{\Phi}^W}(\lambda\partial_d\widetilde{g})&\leq \frac{1}{d}\sum\limits_{j=1}^d \int_{V_{(0,1)}}\Phi (x,d\lambda |\partial_j u(x)|)dx +\frac{1}{d}\sum\limits_{j=1}^d \int_{V_{(0,1)}}\Phi (x,d\lambda |\partial_j u(x)|)dx<\infty.
\end{align*}

Since $\mathcal{R'}$ is a continuous involution it maps compact sets to compact sets. Thus $\widetilde{g}$ is compactly supported. Notice now that 
\[\widetilde{u}=\widetilde{g}\circ T=G_T\widetilde{g},\]
where $G_T$ is the operator from Lemma \ref{Phi i rigid motion}. Hence we conclude that $\widetilde{u}\in W^{1,\Phi_V}(V)$. By the fact that $T$ is a rigid motion we also notice that $\widetilde{u}$ is compactly supported. This completes the proof.

\end{proof}
\begin{theorem} \rm{(}Partition of unity\rm{)}\cite[Theorem 3.15, p. 65]{Adams}
Let $A$ be a subset of $\R^d$ and $\mathcal{U}$ be an open covering of $A$. A family of functions $\{\psi_U\}_{U\in\mathcal{U}}\subset C_C^\infty(\R^d)$ is called a partition of unity associated with the covering $\mathcal{U}$ of $A$, if
\begin{enumerate}
    \item[$(1)$] For every $U\in\mathcal{U}$ and $x\in\R^d$, $0\leq\psi_U(x)\leq 1$.
    \item[$(2)$] For any compact subset $K\subset A$ all but finitely many $\psi_U$ vanish identically on $K$.
    \item[$(3)$] For any $U\in\mathcal{U}$ we have $\supp \psi_U\subset U$.
    \item[$(4)$] For every $x\in A$, $\sum\limits_{U\in\mathcal{U}}\psi_U(x)=1$.
\end{enumerate}
\end{theorem}
Recall the following approximation results proved in the first part \cite{KamZyl3}.

\begin{corollary}\cite[Theorem 3.6]{KamZyl3}\label{cor smooth cutoff}
Let $\Phi$ be a MO function on an open subset $\Omega\subset\R^d$ satisfying  $\Delta_2$ condition and $k\in\N$. For any $f\in W^{k,\Phi}(\Omega)$ we have
\[\lim\limits_{R\to\infty}\|f-f_R\|_{W^{k,\Phi}}=0.\]
\end{corollary}

\begin{corollary}\label{norm approximation of compactly supported functioms}\cite[Corollary 5.6]{KamZyl3}
Let $k\in\N$, $\Omega\subset\R^d$ be open and $\Phi$ be a MO function defined on $\Omega$ satisfying both {\rm (A1)} and $\Delta_2$ conditions. For every compactly supported $f\in W^{k,\Phi}(\Omega)$ such that $\esssupp f\subset \Omega$ there exists a sequence of smooth functions $\{u_n\}_{n=1}^\infty\subset C_C^\infty(\Omega)$ such that 
\[\lim\limits_{n\to\infty}\|f-u_n\|_{W^{k,\Phi}}=0.\]

\end{corollary}

Finally we are able to prove  the main result of this paper.

\begin{theorem}\label{main result of 6}
 Let $\Omega$ be an open set in $\R^d$ such that its boundary is of class $C^1$. Let $\Phi$ be a MO function on $\Omega$ satisfying both  {\rm (A1)} and $\Delta_2$ properties. Then the set of restrictions of functions from $C_C^\infty(\R^d) $ to $\Omega$ is dense in $W^{1,\Phi}(\Omega)$.
\end{theorem}

\begin{proof}
  Let $\Phi$, $\Omega$  be as in the assumption.  Take any $f\in W^{1,\Phi}(\Omega)$.  Let $s:\R^d\to [0,1]$ be a smooth compactly supported function such that $s(x)=1$ for $|x|\leq 1$ and $s(x)=0$ for $|x|>2$. For $R>0$ define $s_R(x)=s\left(\frac{x}{R}\right)$. For any $R>0$ recall that
  \[f_R:=fs_R\]
  and notice that $f_R$ has compact support. Since $\Phi$ satisfies $\Delta_2$, by Corollary \ref{cor smooth cutoff},
  \[\lim\limits_{R\to\infty}\|f-f_R\|_{W^{1,\Phi}}=0.\]
  Hence it suffices to show that the claim of the theorem holds for compactly supported $f\in W^{1,\Phi}(\Omega)$. We have two cases.
  
  $1^0$ $\esssupp f\subset\Omega$.
  
  $2^0$  $\esssupp f\cap\bd(\Omega)\neq\emptyset$.
  
   In the first case, the assumptions of Corollary \ref{norm approximation of compactly supported functioms} are satisfied. Hence there exists a sequence $\{u_n\}_{n=1}^\infty\subset C_C^\infty(\Omega)\subset C_C^\infty(\R^d)$ such that
  \[\lim\limits_{n\to \infty}\|f-u_n\|_{W^{1,\Phi}(\Omega)}=0,\]
  so the claim follows.
  
    In the second case we proceed as follows. Define the set
 \[(\bd(\Omega))_f=\bd(\Omega)\cap\esssupp f.\]
 Clearly $(\bd(\Omega))_f$ is compact. For each $x\in(\bd(\Omega))_f$ let $(V,\{V_{(t_0,t_1)}\}_{(t_0,t_1)\subset (-1,1)},\mathcal{R})$  be the triple from Theorem \ref{reflection property}. We define now $V_x=V$. By compactness of $(\bd(\Omega))_f$, there exists a natural number $K$ and points $x_k\in (\bd(\Omega))_f$, for $k=1,\dots,K$, such that 
 \[(\bd(\Omega))_f\subset\bigcup\limits_{k=1}^K V_{x_k}.\]
 For each $k=1,\dots, K$ define $V^k=V_{x_k}$, then the set $\esssupp f\setminus\bigcup\limits_{k=1}^KV^k$ is a compact subset of $\Omega$. Hence, there exists an open set $V^0$ such that $\overline{V^0}$ is compact and $\overline{V^0}\subset \Omega$ and $\esssupp f\setminus\bigcup\limits_{k=1}^KV^k\subset V^0$. We conclude that $\esssupp f\subset \bigcup\limits_{k=0}^K V^k$. Let $\{\psi_k\}_{k=0}^K$ be the partition of unity associated with the covering $\{V^k\}_{k=0}^K$ of $\esssupp f$. For each $k=0,\dots,K$ we define
 \[f^k=f\psi_k.\]
Clearly $f=\sum_{k=0}^K f^k$
 on $\esssupp  f$. Notice that $f-\sum_{k=0}^K f_k=0 \text{ a.e. on } \Omega\setminus \esssupp f$. Hence
  \[f=\sum_{k=0}^K f^k \text{ a.e. on }\Omega.\]
 Also, for each $k=0,\dots, K$ the function $f^k$ is an element of $W^{1,\Phi}(\Omega)$. Indeed, for $i=1,\dots, d$ and for a.e. $x\in\Omega$,
 \begin{align*}
     |\partial_i f^k(x)|=|\partial_i f(x)\psi_k(x)+ f(x)\partial_i\psi_k(x)|\leq  |\partial_i f(x)|+ C_{i,k}|f(x)|, 
 \end{align*}
  where  $C_{i,k}=\sup\limits_{x\in \R^d}|\partial_i\psi_k(x)|<\infty$. Therefore, for each $k=0,\dots, K$ and $i=1,\dots,d$ we have
  \begin{align*}
    I_{\Phi}(\partial_i f^k) &=\int\limits_\Omega\Phi(x,|\partial_if_k(x)|)dx\leq  \int\limits_\Omega\Phi(x, |\partial_i f(x)|+ C_{i,k}|f(x)|)dx\\&\leq  \frac{1}{2}\int\Phi(x,2 |\partial_i f(x)|)dx+\frac{1}{2}\int\limits_\Omega\Phi(x,2 C_{i,k}|f(x)|)dx<\infty,
  \end{align*}
  by $\Delta_2$ condition.
  
    Since $\esssupp f^0\subset \overline{V^0}\subset \Omega$, $f^0$ is a compactly supported element of $W^{1,\Phi}(\Omega)$, so it satisfies the assumption of case $1^0$ and so there exists a sequence $\{u_{n,0}\}_{n=1}^\infty\subset C_C^\infty(\R^d)$ such that

\begin{align}\label{++}
  \lim\limits_{n\to\infty}\|f^0-u_{n,0}\|_{W^{1,\Phi}(\Omega)}=0. 
\end{align}

 Fix $k\in\{1,\dots,K\}$ and let $\{V^k_{(t_0,t_1)}\}_{(t_0,t_1)\subset(-1,1)}$ be the family of open sets from Theorem \ref{reflection property} and $\mathcal{R}_k$ be the map from the same theorem, both associated with the set $V^k$. Define
 \[\Phi_k: V^k\times[0,\infty)\to[0,\infty)\]
 by the formula
 \[\Phi_k(x,t)=
   \begin{cases} 
      \Phi(x,t) & x\in V^k_{(0,1)} \\
      0 &  x\in\bd(\Omega)\cap V^k \\
      \Phi(\mathcal{R}_k(x),t) & x\in V^k_{(-1,0)}. 
   \end{cases}
\]
By Lemma \ref{extension of Phi} each function $\Phi_k$ is a MO function on $V^k$ and satisfies (A1) and $\Delta_2$ property. We proceed by extending $f^k$ to $V^k$. Define
 \[\widetilde{f^k}(x):=
   \begin{cases} 
      f^k(x) & x\in V^k_{(0,1)} \\
      0 &  x\in\bd(\Omega)\cap V^k \\
      f^k(\mathcal{R}_k(x)) & x\in V^k_{(-1,0)}.
   \end{cases}
\]
 By Theorem \ref{local extension} we conclude that $\widetilde{f^k}$ is compactly supported and
$\widetilde{f^k}\in W^{1,\Phi_k}(V^k)$. By Collorary \ref{norm approximation of compactly supported functioms}, there exists a sequence $\{u_{k,n}\}_{n=1}^\infty\subset C_C^\infty(V^k) $ such that
\begin{align}\label{+++}
    \lim\limits_{n\to\infty}\|\widetilde{f^k}-u_{n,k}\|_{W^{1,\Phi_k}(V^k)}=0.
\end{align}

Moreover, notice that if we extend $u_{n,k}$ to be $0$  in $\R^d\setminus \esssupp u_{n,k}$, this extension is of the class  $C_C^\infty(\R^d)$. Then for any $\lambda>0$ and multi-index $\alpha$ with $|\alpha|\leq 1$ and $k=1,\dots, K$ we have
\begin{align*}
 & I_\Phi (\lambda\partial^\alpha(f^k-u_{n,k|\Omega}))=\int\limits_\Omega\Phi(x,\lambda|\partial^\alpha f^k(x)-\partial^\alpha u_{n,k}(x)|)dx \\
 &= \int\limits_{V^k_{(0,1)}}\Phi(x,\lambda |\partial^\alpha f^k(x)-\partial^\alpha u_{n,k}(x)|)dx \\& \leq  \int\limits_{V^k_{(0,1)}}\Phi(x,\lambda |\partial^\alpha f^k(x)-\partial^\alpha u_{n,k}(x)|)dx +\int\limits_{V^k_{(-1,0)}}\Phi(\mathcal{R}(x),\lambda |\partial^\alpha \widetilde{f^k}(x)-\partial^\alpha u_{n,k}(x)|)dx \\&= \int\limits_{V^k}\Phi_k(x,\lambda |\partial^\alpha \widetilde{f^k}(x)-\partial^\alpha u_{n,k}(x)|)dx =I_{\Phi_k}(\lambda \partial^\alpha(\widetilde{f^k}-u_{n,k})).
\end{align*}
Hence we have for $k=1,\dots,K$,
\begin{align}\label{nierownosc na norme rozszerzenia}
 \|f^k-u_{n,k|\Omega}\|_{W^{1,\Phi}(\Omega)}\leq\|\widetilde{f^k}-u_{n,k}\|_{W^{1,\Phi_k}(V^k)}.   
\end{align}
Define now $u_n=\sum\limits_{k=0}^K u_{n,k}$, $u_n\in C_C^\infty(\R^d)$ and by (\ref{++}), (\ref{+++}), (\ref{nierownosc na norme rozszerzenia}) we have
\[\|f-u_{n|\Omega}\|_{W^{1,\Phi}(\Omega)}\leq\sum\limits_{k=0}^K\|f^k-u_{n,k|\Omega}\|_{W^{1,\Phi}(\Omega)}\leq \sum\limits_{k=0}^K\|\widetilde{f^k}-u_{n,k}\|_{W^{1,\Phi_k}(V^k)}\]
and since the last expression goes to $0$ as $n\to\infty$, we conclude that $\{u_n\}_{n=1}^\infty$ is the desired sequence.  
\end{proof}

We finish with a corollary on density in variable exponent Sobolev spaces $W^{1, p(\cdot)}(\Omega)$. 

\begin{corollary}\label{cor: result of 6}
 Let $\Omega$ be an open set in $\R^d$ such that its boundary is of class $C^1$. Let $\Phi(x,t) =t^{p(x)}$, $t\ge 0$, $p(x) \ge 1$ a.e. in $\Omega$.  If $p(x)$ is essentially bounded on $\Omega$ and satisfies log-H\"older condition \rm{(}see (4.1) in \cite{KamZyl3}\rm{)} then the set of restrictions of functions from $C_C^\infty(\R^d) $ to $\Omega$ is dense in $W^{1,p(\cdot)}(\Omega)$.
\end{corollary}

\begin{proof} It is well known that the exponent variable MO function satisfies condition $\Delta_2$ if and only if the exponent $p(x)$ is essentially bounded on $\Omega$ \cite{Hasbook, zyl}. Moreover condition (A1) is equivalent to log-H\"older condition by Remark 4.3   in \cite{KamZyl3}. Thus the conclusion follows from Theorem \ref{main result of 6}.

\end{proof}

\end{section}

\end{document}